\documentclass[12pt]{amsart}

\usepackage{cancel}
\usepackage{latexsym, amssymb, amsmath}
\usepackage{soul,esint}
\usepackage{amsfonts, graphicx}
\usepackage{graphicx,color}
\usepackage{mathtools}
\usepackage{hyperref}
\usepackage{verbatim}
\makeatletter
\renewcommand*{\eqref}[1]{%
  \hyperref[{#1}]{\textup{\tagform@{\ref*{#1}}}}%
}
\makeatother

\newcommand{\be}{\begin{equation}}
\newcommand{\ee}{\end{equation}}
\newcommand{\beq}{\begin{eqnarray}}
\newcommand{\eeq}{\end{eqnarray}}

\newtheorem{thm}{Theorem}[section]
\newtheorem{conj}{Conjecture}[section]

\newtheorem{lma}{Lemma}[section]
\newtheorem{prop}{Proposition}[section]
\newtheorem{cor}{Corollary}[section]
\newtheorem{defn}{Definition}[section]
\theoremstyle{remark}
\newtheorem{rem}{Remark}[section]
\numberwithin{equation}{section}

\def\be{\begin{equation}}
\def\ee{\end{equation}}
\def\bee{\begin{equation*}}
\def\eee{\end{equation*}}
\def\ol{\overline}
\def\lf{\left}
\def\ri{\right}

\def\wt{\widetilde}
\def\la{\langle}
\def\ra{\rangle}
\def\p{\partial}

\def\ol{\overline}

\def\tr{\operatorname{tr}}

\def\e{\epsilon}
\def\a{{\alpha}}

\def\R{\mathbb{R}}

\def\o{\overline}

\def\d{\frac {\operatorname d}{\operatorname {dt}}}

\def\vh{\vspace{.2cm}}
\def\S{\Sigma}
\def\SS{\partial\Omega}

\def\Lp{\Delta}
\def\Na{\nabla}
\def\eps{\varepsilon}

\def\d{\delta}
\def\po{\partial \Omega}

\def\u{\underline}
\def\MS{\mathcal{S}}
\def\Me{M_{ext}}

\def\MW{\mathcal{W}}
\def\O{\Omega}

\def\ti{\tilde}
\def\MF{\ti M\setminus F}

\def\M{\mathcal{M}}

\def\nas{\nabla^{\Sigma}}
\def\va{\mathbf{a}}
\def\ua{u_{\va}}

\def\Sp{\mathbb{S}^2}
\def\nas{\nabla^{\Sigma}}
\def\mE{\mathcal{E}}

\begin{document}
\title[]
{On A Spacetime Positive Mass Theorem with Corners}

 \author{Tin-Yau Tsang}
\address [Tin-Yau Tsang] {Department of Mathematics, University of California, Irvine}
\email{tytsang@uci.edu}


\date{Feb, 2022}

\begin{abstract} 
In this paper we consider the positive mass theorem for general initial data
sets satisfying the dominant energy condition which are singular across a piecewise smooth
surface. We find jump conditions on the metric and second fundamental form which
are sufficient for the positivity of the total spacetime mass. Our method extends that of
\cite{HKK} to the singular case (which we refer to as initial data sets with corners) using some ideas
from \cite{HMT}. As such we give an integral lower bound on the spacetime mass and we characterise 
the case of zero mass. Our approach also leads to a new notion of quasilocal mass which we show to be
positive and satisfy localised Penrose-type inequality, extending the work of \cite{ST} and \cite{ST3} to the spacetime case. Moreover,  we give sufficient conditions 
under which spacetime Bartnik data sets cannot admit a fill-in satisfying the dominant energy condition. 
This generalises the work of \cite{SWWZ} and \cite{SWW} to the spacetime setting. 
\end{abstract}


\maketitle

\markboth{TIN-YAU TSANG}{SPACETIME POSITIVE MASS THEOREM WITH CORNERS}
 \section{Introduction} 
The positive mass theorem states that an asymptotically flat initial data set $(M^3,g,k)$ satisfying the dominant energy condition must have positive ADM mass. For smooth initial data sets, for the time symmetric case ($k\equiv0$), or sometimes called the Riemannian case, the result has been proved by Schoen and Yau in \cite{SY1} by the minimal surface approach. It is later proved by Witten in \cite{W} assuming the manifold is spin, which is satisfied if the dimension of manifold is 3. For initial data sets which are not time-symmetric, Witten's spinor argument is still applicable (\cite{PT}). Schoen and Yau in \cite{SY2} proved the positive energy theorem through the Jang equation. The positive mass theorem without the spin 
assumption was proved in higher dimensions by Eichmair, Huang, Lee and Schoen in \cite{EHLS} by considering marginally outer trapped surfaces (MOTS) which are analogous to minimal surfaces in the time-symmetric case. Recently, D. Stern (\cite{S}) proposed a new method to study scalar curvature and 3 dimensional topology by considering level sets of harmonic forms. This has given rise to alternative proofs of the positive mass theorem for both the Riemannian case (\cite{BKKS}, cf.\cite{Bartnik}) and the spacetime case (\cite{HKK}) (cf. \cite{J}, \cite{JK}).  We refer readers to the comprehensive survey \cite{BHKKZ} by Bray, Hirsch, Kazaras, Khuri and Zhang for applications of the level set method to the study of ADM mass.  

\

If $k\equiv 0$, for non-smooth initial data sets, the problem has also been extensively studied (\cite{L}, \cite{LL}, \cite{LM}, \cite{MS}, \cite{M1}, \cite{ST}, \cite{ST2}). By generalising the integral formula in \cite{S} and \cite{BKKS}, it is shown in \cite{HMT} that, if the metric is not smooth across a closed hypersurface (corner), then an integral of the mean curvature difference would determine a lower bound of the mass. For the spacetime case, as motivated by the Hamiltonian formulation (\cite{BY}, \cite{HH}), the positive mass theorem with non-compact boundary has been studied in \cite{ADLM}. This also suggests conditions to be imposed on the corner for the positive mass theorem. Together, with the approach in \cite{HMT}, we generalise the integral formula in \cite{HKK} and get the following results. 

\begin{thm}\label{main}
Let $M^3$ be a complete non-compact smooth manifold and $\tilde\Sigma\subset M$ be a piecewise smooth surface.  Assume the metric $g$ and the symmetric (0,2)-tensor $k$ on $M$ satisfy the following:
\begin{enumerate}
    \item $(g, k)$ is  asymptotically flat, 
    \item $g$ and $k$ are smooth up to each component of $\tilde\Sigma$, 
    \item $g$ is Lipschitz, 
    \item $k$ need not be continuous across $\tilde \Sigma$. 
\end{enumerate} 
Let $\mathcal{E}$ be an asymptotically flat end of $M$. Assume there exists $\mathcal{S}$, a finite (possibly empty) disjoint union of connected weakly trapped surfaces which do not intersect $\tilde \S$, such that $H_2(M_{ext}, \MS, \mathbb{Z})=0$, where $M_{ext}$ is the exterior region of $M$ containing $\mathcal{E}$ with $\partial \Me= \MS$. Denote $\tilde\Sigma \cap \Me$ by $\Sigma$. Then for $\mathcal{E}$, there exists a spacetime harmonic function $u$ such that
\begin{equation}\label{ADMLB}
\begin{split}
16\pi (E-|P|)\ge &\int_{M_{ext}\setminus \S}\left(\frac{|\o \Na \o \Na u|^2}{|\nabla u|}+2(\mu|\nabla u|+ \la J, \Na u \ra) \right)\\&+2\int_\S(H_--H_+)|\nabla u|-2\int_\S (\pi_--\pi_+)(\Na u, \nu), 
\end{split}
\end{equation}
where $\pi_\pm$ and $H_\pm$ respectively denote the conjugate momentum tensors of $k_{\pm}$  and the mean curvatures of $g_\pm$ on $\S$ with respect to $\nu$, the unit normal pointing into the infinity of $\mathcal{E}$. 
In particular, if the dominant energy condition holds on $\Me \setminus \S$ and $$(H_--H_+)-|\omega_--\omega_+|\geq0$$ on $\S$, then we have 
$$
E\geq|P|,
$$
where $\omega_{\pm}:=\pi_{\pm}(\cdot,\nu)$. 
\end{thm}

\begin{cor}\label{rigid}
Assume the dominant energy condition holds on $\Me\setminus \S$ and $$(H_--H_+)-|\omega_--\omega_+|\geq0$$ on $\S$. If $E=|P|$, then $M$ is diffeomorphic to $\R^3$. If $E=|P|=0$, then $(M,g,k)$ arises from an isometric embedding into Minkowski space as the graph of a linear combination of spacetime harmonic functions.  
\end{cor}

\begin{cor}\label{regids}
Assume that $\S$ is smooth, the dominant energy condition holds on $\Me\setminus \S$ and $$(H_--H_+)-|\omega_--\omega_+|\geq0$$ on $\S$. If $E=|P|=0$, $k$ is continuous and the normal derivative of $tr_g k$ is continuous on $\S$, then $g\in C^{2,1}_{loc}(M)$ and $k\in C^{1,1}_{loc}(M)$. 
\end{cor}

In particular, Theorem \ref{main} holds for manifolds with vanishing second homology which imposes control on the topology of level sets of $u$. For smooth initial data sets, the construction of $M_{ext}$ is shown in \cite{HKK} if the dominant energy condition holds on $M$. Their proof adopts \cite{ADGP} to show that an open subset of $M$ containing an asymptotically flat end admits PSC topology by the density theorem of \cite{EHLS}, together with the existence, regularity and compactness properties of apparent horizons (\cite{E1}, \cite{E2}, \cite{E3}).  If an initial data set admits singularities, the density theorem and the existence of marginally trapped surfaces are not known. Therefore, it would be of interest to study the Jang equation on singular initial data sets.

\

The positive mass theorem with corners is key to the study of properties (mass, geometry, etc.) of compact initial data sets by connecting them to those of asymptotically flat ones.  If a physical system is not isolated or cannot be viewed from infinity where asymptotic symmetry exists, e.g. compact manifolds with boundary, the ADM mass is not well defined. Different notions of quasilocal mass have been suggested. By the Hamiltonian formulation, the Brown-York mass (\cite{BY}) was proposed for the time-symmetric case. Shi and Tam in \cite{ST} have constructed a quasispherical asymptotically flat extension for compact manifolds which is scalar flat and has prescribed mean curvarture on the boundary, then they applied the Riemannian positive mass theorem for Lipschitz metrics to prove that the Brown-York mass is positive. For the spacetime case, the Liu-Yau mass (\cite{LY1}, \cite{LY2}, see also \cite{KI}) and the Wang-Yau mass (\cite{WY2}) which generalise the Brown-York mass are proved to be positive by considering the Jang graph over a compact initial data set with the dominant energy condition and applying Shi-Tam's result. 

\

A natural question arises, that is whether the spacetime positive mass theorem with corners can help show positivity of some quasilocal quantities. Together with \cite{MST}, we have the following results.  

\begin{cor}\label{qm}
Let $(\Omega^3, g, k)$ be a compact initial data set satisfying the dominant energy condition. Assume there exists $\MS$, a finite (possibly empty) disjoint union of connected weakly trapped surfaces, such that $H_2(\Omega_{ext}, \MS, \mathbb{Z})=0$, where $\Omega_{ext}$ denotes the portion of $\Omega$ outside $\MS$. Suppose $\S=\SS$ is a smooth surface with finitely many components with Gaussian curvature $\kappa>0$ and mean curvature $H$ with respect to the outward normal $\nu$. Denote the mean curvature of an isometric embedding of $\S$ into $\R^3$ with respect to the outward normal by $H_0$. If $H > |\omega|$, where $\omega=\pi(\cdot,\nu)$, then 
$$\MW(\S):=\frac{1}{8\pi}\int_\S H_0-(H-|\omega|) \geq 0.$$ 
If $\MW(\S)=0$, then $\S$ is connected, $\Omega$ is diffeomorphic to a domain in $\R^3$ and can be isometrically embedded into Minkowski space. 
\end{cor} 

\begin{cor}
Let $(\Omega^3, g, k)\subset \R^{3,1}$ be a compact initial data set.  If  $\S=\p \O$ is smooth and connected, has positive Gaussian curvature and $H>|tr_{\S}k|$, where $H$ is the mean curvature of $\S$ with respect to the outward normal, then $\MW(\S)\geq 0$.  Moreover, $\MW(\S) = 0$ if and only if $\O$ is a domain within a hyperplane in $\R^{3,1}$. 
\end{cor}

In general relativity, (apparent) horizons are intriguing objects. The Penrose inequality which suggests the quantitative relations between the ADM mass and the area of horizons is proved for the Riemannian case in \cite{HI} and \cite{B}.  Furthermore, the hoop conjecture initiated a search for suitable quantities to detect the existence of (apparent) horizons. Schoen and Yau in \cite{SY3} and Yau in \cite{Y} defined a kind of radius of a bounded region. Under certain assumptions on the energy condition, they have shown an upper bound of the radius of the region in which apparent horizons cannot exist. For the time-symmetric case, Shi and Tam in \cite{ST3} showed that the Brown-York mass satisfies a comparison theorem with the Hawking mass and a localised Penrose inequality. Together with the Shi-Tam mass, these indicate sufficient conditions in which horizons exist in the spirit of the hoop conjecture. Recently, the result was further generalised in \cite{ALY} for the spacetime case for the Liu-Yau mass and the Wang-Yau mass. As inspired by their methods, we will show that $\MW(\S)$ also satisfies a comparison theorem with the Hawking mass. And as a corollary, we obtain a localised Penrose-type inequality. These two results can help reveal the internal geometry, in particular the existence and non-existence of minimal surfaces of an initial data set. 

\begin{thm}\label{CTI}
Let $(\Omega^3,g,k)$ be admissible\footnote{See Section \ref{Applications of W} for the definition.}. Then, for any connected minimising hull $E$ in $\Omega$ where $E\subset \subset \Omega$ with $C^{1,1}$ boundary $\p E$, we have 
$$\MW(\S)\geq m_H(\p E),$$
where $m_H$ stands for the Hawking mass. 
\end{thm}
\begin{cor}
Let $(\O^3,g,k)$ be admissible. Suppose that $S$ is an outward minimising surface in $\Omega$, then 
$$\MW(\S)\geq\sqrt{\frac{|S|}{16\pi}}.$$
\end{cor}

Moreover, by comparing $\MW(\S)$ with the Liu-Yau mass, we can see that the localised spacetime Penrose inequality and relevant results of MOTS in \cite{ALY} also holds for $\MW(\S)$ if $(\Omega,g,k)$ satisfies certain admissibility conditions.  

\

For geometry on compact manifolds with boundary, Gromov proposed the following conjecture. 
\begin{conj}\label{Gromov}(\cite{Gro19} Sect 3.12.2 III., IV.) Let $(M, g)$ be a compact Riemannian manifold with scalar curvature $R \geq \sigma$. Then there exists $\Lambda$ depending only on $\sigma$ and the intrinsic geometry of $(\p M, g|_{T(\p M)})$ such that
\be
\int_{\p M} H \leq \Lambda,
\ee
where $H$ is the mean curvature of the boundary $\p M$ in $(M, g)$ with respect to the outward unit normal vector.  
\end{conj}

In \cite{SWWZ} and \cite{SWW}, there was a partial affirmative answer given by the parabolic method to construct a metric extension done in \cite{ST}.  It is shown that if the mean curvature is too large and an NNSC (non-negative scalar curvature) fill-in to a Bartnik data set exists, then there is a contradiction to the Riemannian positive mass theorem with corners. As motivated by the idea of Hamiltonian formulation, with Theorem \ref{main}, we can show non-existence of DEC fill-ins for spacetime Bartnik data sets (\cite{B2}) if the boundary energy is too large.    
\begin{thm} 
Let $D_{SB}:=(\S^2, \gamma, \alpha, H, \beta )$ be a spacetime Bartnik data set where $\S^{2}$ can be embedded into $\R^3$ and $\gamma$ is smooth. There exists a constant $C_0=C_0(\Sigma,\gamma)>0$ such that if $$H-f\geq C_0,$$ 
where $f:=\sqrt{(\tr_{\S}\alpha)^2+|\beta|_{\gamma}^2}$,
then $D_{SB}$ cannot admit a fill-in satisfying both of the following: 
\begin{enumerate}
\item there exists $\MS$, a finite (possibly empty) disjoint union of connected weakly trapped surfaces, such that $H_2(\Omega_{ext}, \MS, \mathbb{Z})=0$, where $\Omega_{ext}$ denotes the portion of $\Omega$ outside $\MS$, 
\item the dominant energy condition.
\end{enumerate}
\end{thm}

\begin{rem} 
While assuming that the initial data sets are spin, Shibuya (\cite{Shibuya}) proved a spacetime positive mass theorem with data of lower regularity than that imposed in Theorem \ref{main} by the idea of distributional curvature proposed in \cite{LL}.  In particular, \cite{Shibuya} can imply positive mass theorem on 3 dimensional initial data sets with corners without restriction on second homology. Furthermore, some of the aforementioned results can be extended to higher dimensions if we assume the initial data sets are spin. These will be addressed in the body of the paper.  
\end{rem}

\begin{rem}
All manifolds in this paper are assumed to be orientable. 
\end{rem}

This text is organised as follows. In Section \ref{Preliminaries}, asymptotically flat initial data sets and the assumptions on corners are discussed. 
In Section \ref{Spacetime harmonic coordinates}, the existence and regularity of spacetime harmonic functions is derived. In Section \ref{Regular level set topology}, we discuss the topology of regular level sets of spacetime harmonic coordinates. In Section \ref{Boundary formula}, the integral formula for spacetime harmonic functions with general boundary conditions is derived. In Section \ref{SPMTC}, we will prove Theorem \ref{main} by applying the integral formula. In Section \ref{Embedding into Minkowski space} and \ref{Raising Regularity}, the proof of Corollary \ref{rigid} and \ref{regids} are presented respectively. In Section \ref{Connections to quasilocal mass}, we introduce the quantity $\MW(\S)$. In Section \ref{Applications of W}, we
prove the comparison theorem for $\MW(\S)$ and show its applications in detecting minimal surfaces. The non-existence of DEC fill-ins is discussed in Section \ref{DECfillin}.  

\

\textbf{Acknowledgements }: The author would like to thank Daniel Stern for kindly answering his questions and clearly explaining \cite{BKKS}. The author would like to thank Demetre Kazaras for explaining their work in \cite{HKK} in detail. The author also thanks Pak-Yeung Chan, Sven Hirsch, Man-Chun Lee, Chao Li and Long-Sin Li for helpful conversations. The author would like to thank Prof. Pengzi Miao for kindly answering his questions. The author is grateful to Prof. Martin Man-Chun Li, Prof. Connor Mooney, Prof. Richard Schoen, Prof. Luen-Fai Tam and Prof.  Mu-Tao Wang for insightful discussions.

\section{Preliminaries}\label{Preliminaries}
\subsection{Asymptotically flat initial data sets}
Given an initial data set \\ $(M^3,g,k)$, where $g$ is a Riemannian metric and $k$ is a symmetric $(0,2)$-tensor. Define the conjugate momentum tensor by $\pi=k-(tr_g k)g$. Under constraint equations, we can define the mass density $\mu$ and the current density $J$ by
$$\mu=\frac{1}{2}(R_g+(tr_g k)^2-|k|_g^2)$$
and $$J=div_g(k-(tr_gk)g)=div_g\pi.$$
$(M,g,k)$ is said to satisfy the dominant energy condition if 
$$\mu\geq|J|_g.$$ 
We say $(M, g, k)$ is asymptotically flat if there exists a compact set $\mathcal{C}\subset M$ such that  $M\setminus \mathcal{C}=\coprod_{i=1}^k N_i$, where each end $N_i=\R^3\setminus B_{r_i}$ through a coordinate diffeomorphism in which 
$$
g_{ij} = \delta_{ij} + O^2(|x|^{-q}),
$$
and 
$$
k_{ij}=O^1(|x|^{-q-1}),
$$
where $q>\frac{1}{2}$, $ \mu, J \in L^1(M)$ and for a function $f$ on $M$, $f = O^{m}(|x|^{-p})$ means $\sum_{|l|=0}^m||x|^{p+|l|}\partial^lf|$ is bounded near the infinity.

\

For each end, the ADM energy-momentum vector $(E,P)$ and the ADM mass $\mathfrak{m}$ \cite{ADM} are given by 
$$
E = \frac{1}{16\pi} \lim_{r \to \infty} \int_{|x| = r}  ( g_{ij,i} - g _{ii, j} ) \nu^j,
$$ 
$$P_i:=\frac{1}{8\pi} \lim_{r \to \infty} \int_{|x| = r}  \pi_{ij} \nu^j, \,\ \,\ \,\ i=1,2,3,$$ 
and 
$$\mathfrak{m}=\sqrt{E^2-|P|^2},$$
where the outward unit normal $\nu$ and surface integral are with respect to the Euclidean metric. 

\

The set up of Theorem \ref{main} is as follows. $\tilde \Sigma\subset M$ we consider is a piecewise smooth surface with (possibly empty) piecewise smooth boundary.  Since we can fill in $\p \ti \S$ by a surface in $M$, hereafter, it is assumed that $\ti \S$ is some open sets' boundary consisting of piecewise smooth surfaces whose boundaries are piecewise smooth curves and vertices, where the dihedral angles between faces are bounded from below by a positive constant.  For example, $\ti \S$ can be the boundary of balls, cylinders, polyhedra and cones in $\R^3$.  In this setting, motivated by Hamiltonian formulation (see Section \ref{Hamiltonian formulation}), Theorem \ref{main} provides partial results on dihedral rigidity for initial data sets in \cite{T}.   

\

Denote the designated end by $\mathcal{E}$.  Let $\Sigma_i$ be a connected component of $\Sigma$.  Let $\nu$ denote the normal on faces of $\Sigma_i$ pointing toward $\mathcal{E}$.  A neigbourhood of $\S_i$ in $M$ on the same side to which $\nu$ is pointing is denoted by $U_+$ while the one on the opposite by $U_-$. The metrics on $U_{\pm}$ induced by $g$ are denoted by $g_{\pm}$ and their mean curvatures on $\Sigma_i$ with respect to $\nu$ are denoted by $H_{\pm}$. Similarly, we can define $k_{\pm}$ and $\pi_{\pm}$ on $\S$. 

\

The regularity assumptions of $(g,k)$ in Theorem \ref{main} naturally arise from the fill-in and extension problems (e.g. \cite{B}, \cite{M1}, \cite {ST}, \cite{SWWZ}, \cite{SWW}).  For example, let $(M_1, g_1)$, $(M_2, g_2)$ be two Riemannian manifolds with smooth boundary, where $\p M_1$ is isometric to $\p M_2$. As mentioned in Section 3 of \cite{M1}, one can respectively identify the Gauss tubular neighbourhoods of $\p M_1$ in $M_1$ and $\p M_2$ in $M_2$ with $U_1=\p M_1 \times(-2\eps,0]$ and $U_2=\p M_2 \times[0,2\eps)$ for some $\eps>0$ by Fermi coordinates $(x,t)$.  Then, $g_1\cup g_2$ would be a continuous metric on the glued manifold $M_1\cup M_2$ under this chart. The smooth structure might be altered but the topology remains the same.

\

For a smooth closed hypersurface $S\subset M$, we say $S$ is a weakly outer trapped surface if on $S$, the outer null expansion 
$$\theta_+=H+tr_S k \leq 0,$$
and a marginally outer trapped surface ($MOTS$) if 
$$\theta_+=0;$$
correspondingly, 
a weakly inner trapped surface if the inner null expansion 
$$\theta_-=H-tr_S k \leq 0,$$
and a marginally inner trapped surface ($MITS$) if 
$$\theta_-=0,$$ 
where $H$ is computed with respect to the normal pointing to the infinity of the designated end $\mathcal{E}$. A surface is weakly trapped if it is either weakly outer trapped or weakly inner trapped. 
In this note, the mean curvature is computed in the convention that $\mathbb{S}^2\subset \R^3$ has positive mean curvature with respect to the outward normal.  
\

If $M$ contains more than one ends, by the decay rate of $g$ and $k$, we know large coordinate spheres in all the ends other than $\mathcal{E}$ satisfy $\theta^+<0$. Therefore, we can assume that $\Me$ has one end $\mathcal{E}$ only. 

\

Since $H_2(\Me,\MS,\mathbb{Z})=0$, we can compartment $M_{ext}$ into different components as follows, $$\Me=M_{0}\cup_{i=1}^l K_i \cup_{j=1}^m \Omega_{j},$$ for some $l, m \geq 0$, where 
\begin{enumerate}
\item
$M_0$ is the component containing $\mathcal{E}$ with the boundary composed of components of $\MS$ and components of $\Sigma$, 
\item 
$K_i$ is compact with the boundary composed only of components of $\Sigma$, 
\item 
$\Omega_{j}$ is compact with the boundary composed of components of $\MS$ and a component of $\Sigma$.
\item 
$(g, k)$ is smooth on each of the components. 
\end{enumerate}

\subsection{Hamiltonian formulation (Hamilton-Jacobi analysis)}\label{Hamiltonian formulation}
Let $(\Omega^n,g,k)$ be a compact initial data set with boundary $\S$.  A spacetime $(N^{n+1},\bar g)$ with boundary $\bar \S$ can be constructed by infinitesimally deforming the initial data set $(\Omega,g,k,\S)$ in a transversal, timelike direction $\p_t= V \vec{n} + W^i\p_i$ which satisfies $\bar\Na_{\p_t} t=1$, where $V$ is the lapse function, $\vec{n}$ is the timelike unit normal of $\Omega$ in $N$ and $W$
is the shift vector. Further assume that $\Omega$ meets
$\bar \S$ orthogonally. The purely gravitational contribution $\mathcal{H}_{grav}$ to the total
Hamiltonian at the slice $\Omega$ is given by (\cite{ADM},\cite{RT},\cite{BY},\cite{HH})
\be \label{Ham}
c(n) \mathcal{H}_{grav}(V,W) = \int_{\Omega} (\mu V  + \la J,W \ra) - \int_{\S}(HV-\pi(\nu,W)),
\ee
where $H$ is the mean curvature of $\S$ with respect to the outward normal of $\Omega$ and $\pi$ is the conjugate momentum tensor.  From this, if we consider an asymptotically flat initial data set with corners, we can expect that the contribution to mass from the corners is due to the mean curvature $H$ and the 1-form $\pi(\nu,\cdot)$.  An interesting difference between \eqref{Ham} and \eqref{ADMLB} is that they are related to a timelike vector field and a null vector field (see Section \ref{Spacetime harmonic coordinates}) respectively.  From this perspective,  expressing the energy by null vector fields on initial data sets is a significant property of the level set approach. 

\subsection{Hyperbolic space patched with negative mass Schwarzschild}
We are going to construct an initial data set which does not satisfy the conditions on corners stated in Theorem \ref{main}.  
Let us consider a rotationally symmetric data set of the form $g=u(r) dr^2+r^2 g_{\mathbb{S}^2}$, where $g_{\mathbb{S}^2}$ is the standard metric on $\mathbb{S}^2$. Let $u=\frac{1}{1+r^2}$ for $0\leq r\leq 1$ and $u=\frac{1}{1-\frac{2m}{r}}$ for $r\geq 1$, take $m=-\frac{1}{2}$ so that $u$ is continuous at $r=1$. 

\

Note that the metric $g_-$ for $r<1$ is the hyperbolic metric while $g_+$ for $r>1$ is the negative mass Schwarzschild metric. This metric $g=(g_-,g_+)$ is then Lipschitz across $\Sigma=\{r=1\}$ with $H_-=H_+$ on $\S$. 
If we take either $k_-=g_-$ or $-g_-$ for $r<1$ and $k_+=0$ for $r>1$, then away from
$\Sigma$ we see that $(g,k)$ satisfies the vacuum constraint equations, $\mu=|J|=0$. Moreover, $H_--H_+-|\omega_--\omega_+|<0$ on $\S$. 

\

For this initial data set,  $E=m=\frac{-1}{2}$ and $|P|=0$. By the definition of ADM energy-momentum vector, we can see that under different choices of $k$, $E-|P|$ is still of the same sign. This tells us the jump of expansions $\theta_{\pm}=H\pm tr_\S k$ would not be a sufficient condition for the spacetime positive mass theorem with corners in general. The example also shows that the negativity of $E-|P|$ can be expected from the conditions on the corner.

\section{Spacetime harmonic coordinates on initial data sets with corners}\label{Spacetime harmonic coordinates}
In this section, we are going to show existence and regularity of spacetime harmonic functions on a non-smooth initial data set. As mentioned by \cite{HKK}, if we consider $(M,g,k)$ as a spacelike slice of a spacetime $(\o M,\o g)$, then for a smooth function $\tilde u$ on $\o M$, for $X,Y\in TM$, the spacetime Hessian $\o \Na\o \Na \tilde u(X,Y)=\Na \Na \tilde u (X,Y) +k(X,Y)\vec{n}(\tilde u)$, where $\vec{n}$ is the timelike unit normal of $M$ in $\o M$. And if $\o \Na \tilde u$ is null, we have $\o \Na \o \Na \tilde u (X,Y)=\Na \Na \tilde u (X,Y) +|\Na \tilde u|k(X,Y)$.
\begin{defn}
A function $u$ on $M$ is called spacetime harmonic if $$\o \Lp u:=tr_g \o\Na \o\Na u=\Lp u+(tr_g k)|\Na u|=0.$$
\end{defn}
Moreover, as motivated by \cite{Bartnik} and \cite{BKKS}, we would further like such $u$ to be asymptotic to the original asymptotically flat coordinates.  

\subsection{Existence and regularity}\label{existence}
Form this section onward, for notation simplicity, write $tr_g k$ by  $K$. Let $\MS$ be a finite (possibly empty) disjoint union of connected weakly trapped surfaces which do not intersect $\ti \S$ such that $H_2(M_{ext},\MS,\mathbb{Z})=0$. Let $\Sigma=\tilde \Sigma \cap \Me$. Following the strategy of Section 4 in \cite{HKK}, we prove the following proposition. 
\begin{prop}
For the asymptotically flat coordinate $x^1$, for any $\phi\in C^{\infty}(\MS)$, there exists $u\in  W^{2,p}_{loc}(\Me)\cap W^{3,p}_{loc}(\Me\setminus \S)$ such that 
\begin{enumerate}
    \item $\Lp u+K|\Na u|=0$ on $M_{ext}$, 
    \item $u=\phi$ on $\MS$, 
    \item $u-x^1=O^2(|x|^{1-q})$ as $|x| \to \infty$, 
    \item $u|_{\S}$ is $C^2$ on faces of $\S$. 
\end{enumerate} 
\end{prop}

\begin{rem} 
In Section \ref{Regular level set topology}, we would discuss that $\phi$ can be chosen to achieve suitable signs on the normal derivative of $u$ on $\MS$. 
\end{rem}

\begin{proof}
By slightly generalising Proposition 2.2 and Theorem 3.1 in \cite{Bartnik}, we can have a function $v\in W^{2,p}_{loc}(M)$, where $p>3$, such that 
\begin{enumerate}
    \item $\Lp v=-K$ on $\Me$, 
    \item $v=0$ on $\MS$, 
    \item $v=x^1+O^2(|x|^{1-q})$ as $|x| \to \infty$.
\end{enumerate}
Note that by elliptic regularity, $v$ is smooth on $\Me\setminus \S$ and $C^{1,\alpha}$ across $\S$ by Sobolev embedding. We also define a compactly supported smooth function $v_0$ such that $v_0=\phi$ on  $\MS$. Define $\tilde{v}=v+v_0$.

\

 Let $r>>1$ and $M_r$ denote the region of $\Me$ enclosed by the coordinate sphere $S_r=\{|       x|=r\}$. Consider the following localised Dirichlet problem, 
\begin{enumerate}
\item $\Lp u^r +K|\Na u^r|=0$ in $M_r$,
\item $u^r=\phi$ on $\MS$, 
\item $u^r=\tilde{v}$ on $S_r$.
\end{enumerate}  
 
Let $w^r=u^r-\tilde{v}$. It is then equivalent to seek existence of $w^r$ which solves,
 \begin{enumerate}
\item $\Lp w^r=-K\left( \frac{\Na(w^r+2\tilde{v})}{|\Na(w^r+\tilde{v})|+|\Na \tilde{v}|} \right)\cdot \Na w^r-\Lp \tilde{v}-K|\Na \tilde{v}|$ in $M_r$, 
\item $w^r=0$ on $\MS$, 
\item $w^r=0$ on $S_r$.
\end{enumerate}  
Construct a map $\mathcal{F}:C_0^{1,\alpha}(M_r)\times [0,1]\to C_0^{1,\alpha}(M_r)$ by 
\be
\mathcal{F}(w, \sigma)=\sigma \Lp^{-1} F(w), 
\ee
where $F:C_0^{1,\alpha}(M_r)\to L^p(M_r)$ is defined by 
$$F(w)=-K\left( \frac{\Na(w+2\tilde{v})}{|\Na(w+\tilde{v})|+|\Na \tilde{v}|} \right)\cdot \Na w-\Lp \tilde{v}-K|\Na \tilde{v}|.$$

In particular, we can see $w^r=\mathcal{F}(w^r,1)$. Consider the following composition (cf. equation (4.10) in \cite{HKK}), 
 \be
C_0^{1,\alpha}(M_r)\xrightarrow{F} L^p(M_r)\xrightarrow{\Lp^{-1}} W^{2,p}(M_r)\cap W^{1,p}_0(M_r)\xrightarrow{\iota} C_0^{1,\alpha}(M_r). 
 \ee 
Note that $F$ and $\Lp^{-1}$ are bounded and the inclusion is compact by Sobolev embedding. Let $w_\sigma=\mathcal{F}(w_\sigma, \sigma)$, we have
\be
\Lp w_{\sigma}+\sigma K\left( \frac{\Na(w_{\sigma}+2\tilde{v})}{|\Na(w_{\sigma}+\tilde{v})|+|\Na \tilde{v}|} \right)\cdot \Na w_{\sigma}=-\sigma \Lp \tilde{v}-\sigma K|\Na \tilde{v}|. 
\ee
Since the zeroth order term coefficient vanishes, maximum principle (Theorem 9.1 in \cite{GT}) is applicable, we can have a uniform $W^{2,p}(M_r)$ apriori estimate for all $w_\sigma$ by Theorem 9.11 and 9.13 in \cite{GT}. Thus, $w_{\sigma}$ is uniformly bounded in $C^{1,\alpha}(M_r)$ by Sobolev embedding.  

\

By Leray Schauder fixed point theorem (\cite{GT} Theorem 11.6), we can seek existence of $w^r$. And by the barrier function of order $O(|x|^{1-2q})$ constructed in Section 4.2 in \cite{HKK} and maximum principle, we can obtain a uniform $W^{2,p}_{loc}$ bound for all $w^r$. Hence, $w_r$ is uniformly $C^{1,\alpha}_{loc}$ bounded. Away from $\Sigma$, we can see that $\Lp w^r\in C^{0,\alpha}$, and hence $w^r\in C^{2,\alpha}_{loc}(\Me\setminus \S)$, uniformly bounded .

\

Hence, by taking a diagonal subsequence as $r\to \infty$, we have a spacetime harmonic function $u:=\lim_{r\to\infty}w^r+\tilde{v}=\tilde{v}+O^2(|x|^{1-2q})=x^1+O^2(|x|^{1-q})$, $u\in C^{1,\alpha}_{loc}(\Me)\cap C^{2,\alpha}_{loc}(\Me\setminus \S)$. Furthermore, since $|\Na u|\in W^{1,p}_{loc}(\Me)$ by Kato's inequality, we have $u \in W^{3,p}_{loc}(\Me \setminus \S)$ by Theorem 9.19 in \cite{GT}.  

\

Then, we are going to consider the regularity of $u$ nearby $\S$. 
Let $\hat \S$ be a smooth surface component of $\S$. 
Let $p\in \hat \S$ and $V$ be a neighborhood of $p$ in $M$ which does not intersect $\S\setminus \hat \S$.  Apply Fermi coordinate along $\hat \S$, $(x^1, x^2, t) \in \S\times (-\eps,\eps)$, considering the difference quotients along $\p_1$ direction, let $\phi^h=\Lp^h u$, where $\Lp^h f(x_1,x_2,t):=\dfrac{f((x_1+h,x_2,t)-f((x_1,x_2,t))}{h}$ for a function $f$. Since $u$ is spacetime harmonic, we have  
\be
\begin{split}
&g^{ij}\phi^h_{ij}-g^{ij}\Gamma_{ij}^k\phi^h_k \\
=& f^h := -\Lp^hg*\widetilde{\p^2 u}+\Lp^h(g*\Gamma)*\widetilde{\p u}-(\Lp^h K)\widetilde{\,|\Na u|\, }-K \Lp^h |\Na u|,\\
\end{split}
\ee
where $*$ denotes multiplication with indices suppressed and $\tilde{\phi}(x_1,x_2,t)=\phi(x_1+h,x_2,t)$ for functions on $V$. Observe that the $g\in C^{0,1}(V)$, $\Gamma\in L^{\infty}(V)$ and $u\in W^{2,p}(V)$. While along $\p_1$ direction, except on a $\mathcal{H}^3$-measure zero set $\S \cap V$, $g$ and $k$ are also smooth. Moreover, $|\Na u|\in W^{1,p}(V)$. Hence, its difference quotient is uniformly bounded in $L^p(V)$. Therefore, $f^h$ is uniformly bounded in $L^{p}(V)$. By Theorem 9.11 in \cite{GT}, we know for any $U\subset\subset V$, $||\phi^h||_{W^{2,p}(U)}$ and hence $||\phi^h||_{C^{1,\alpha}(U)}$  is uniformly bounded. Therefore, we have $\phi:=\lim_{h\to0} \phi^h=\p_1 u \in C^{1,\alpha}(U)$ by \cite{GT} Lemma 7.24. 

\
  
By varying the direction which is tangential to $\hat \S$ and the neighbourhood for difference quotients, we can see the same argument applies. Therefore,  $u|_{\S}$ is $C^2$ on faces of $\S$. 
\end{proof}

\section{Regular level set topology}\label{Regular level set topology}
In this section, we would first discuss the regular level set as a whole in $\Me$. Then we would further study the intersection of the regular level set with the corner $\S$. This is essential for analysis in Section \ref{Boundary formula} and \ref{SPMTC} when we study the boundary terms of the integral formula  (Lemma \ref{IF}). We would denote a level set $\{u=t\}$ by $\S_t$. 

\subsection{Structure of regular level sets in $\Me$}
Denote each component of $\MS$ by $\p_i M$, $i=1,2,...n$.
Let $u_{\vec{c}}$, where $\vec{c}=(c^1, c^2,..., c^n)$ is a constant vector, be a spacetime harmonic function such that 
\begin{enumerate}
\item $\Lp u_{\vec{c}}+K |\Na u_{\vec{c}}|=0$ in $M_{ext}$, 
\item $u_{\vec{c}}=c^i$ on $\p_i M$, 
\item $u_{\vec{c}}=v+O^2(|x|^{1-2q})$ as $x\to\infty$.
\end{enumerate}
Its existence and regularity have been analysed in Section \ref{Preliminaries}.
We are going to show the following 2 conclusions from \cite{HKK} are still valid for the solution we have obtained which is of slightly lower regularity. 
\begin{lma}\label{ND}(\cite{HKK} Lemma 5.1)
Let $a_i \in\{-1,1\}$ for $i=1,2,...,n$. There exists a constant $\vec{c}$ such that for each $i$, there exists $y_i \in \p_i M$ with $|\Na u_{\vec{c}}(y_i)|=0$, and $(-1)^{a_i}(\p_\nu u_{\vec{c}})\geq 0$ on $\p_i M$, where $\nu$ is the unit normal pointing out of $\Me$.
\end{lma}
\begin{thm}\label{top}(\cite{HKK} Theorem 5.2)
Let  $\vec{c}$ be the constant obtained from Lemma \ref{ND}, then all regular level sets of $u_{\vec{c}}$ are connected and non-compact with a single end modeled on $\R^2\setminus B_{1}$. Hence, a regular level set would have Euler characteristic $\leq 1$.  
\end{thm}

It suffices to show that $u_{\vec{c}}$ is continuously differentiable in $\vec{c}$, in the sense of Section 5 in \cite{HKK}, which is as follows. 
\begin{lma}
$\Psi:\R^n \to C^{1,\alpha}(\Me)$ is a $C^1$ map, where $\Psi(\vec{c}):=u_{\vec{c}}-v$ and $v$ is defined as in Section \ref{Spacetime harmonic coordinates}. 
\end{lma}
\begin{proof}
 For simplicity, say $n=1$. Now, we have 2 spacetime harmonic functions $u_t$ and $u_s$, define $w:=u_t-u_s=\Psi(t)-\Psi(s)$, we can see $w$ solve the following Dirichlet problem, 
\begin{enumerate}
\item
$\Lp w -K \left(\frac{\Na u_t+\Na u_s}{|\Na u_t|+|\Na u_s|}\right)\cdot \Na w=0$ in $\Me$, 
\item
$w=t-s$ on $\p M$,
\item $w=O^2(|x|^{1-2q})$ as $|x| \to\infty$.
\end{enumerate}

Let $R>>1$, denote the part of $\Me$ enclosed by coordinate sphere $S_R=\{|x|=R\}$ by $M_R$. Let $\phi_R$ be a function satisfying the boundary conditions $\phi_R=t-s$ on $\p M$ and $\phi_R=w=O(R^{1-2q})$ at $S_R$. We can extend $\phi_R$ into $M_R$ such that $||\phi_R||_{C^0}=|t-s|, |\p^k \phi_R|\leq \frac{C}{R^{k}}$, $k=1, 2$.
Then by Theorem 8.33 in  \cite{GT}, we have 
$$||w||_{C^{1,\alpha}(M_R)}\leq C\left( ||w||_{C^0(M_R)}+||\phi_R||_{C^{1,\alpha}(M_R)} \right).$$
                                                                    
Note that the coefficient on the zeroth order term is zero and hence maximum principle (\cite{GT} Theorem 9.1) can be applied. Then we know $||w||_{C^0(M_R)}=||w||_{C^0(\p M_R)}=|t-s|$. Therefore, we have, 
    $$||w||_{C^{1,\alpha}(M_R)}\leq C\lf( |t-s|+\frac{C}{R} \ri).$$ 
Take $R\to \infty$, we have. 
$$||w||_{C^{1,\alpha}(\Me)}\leq C(|t-s|).$$
Therefore,  
$$\frac{\Psi(t)-\Psi(s)}{t-s}$$ converges subsequently as $t\to s$. Hence, $\Psi$ is differentiable in $c$.

\

Further note that $\p_c v=0$, define $u_c^{'}=\p_c u_c=\p_c \Psi$, then we have (equations (5.3) and (5.4) in \cite{HKK}), 
\begin{enumerate}
\item $\Lp u_{c}^{'}+K \frac{\Na u_c}{|\Na u_c|}\cdot \Na u_c^{'}=0$ in $M_{ext}$, 
\item $u_{c}^{'}=1$ on $\p M$,  
\item $u_{c}^{'}=O(|x|^{1-2q})$ as $|x| \to\infty$.
\end{enumerate}
Note that, for all $c$, $u_c^{'}$ are bounded by 1 in $L^{\infty}$ by maximum principle and satisfy a PDE with uniformly bounded coefficients. Therefore, they have uniform $W^{2,p}_{loc}$ bound. In particular, $||\Na u_c^{'}||_{L^p_{loc}}$ are uniformly bounded.

\

Fix $t$, for all $s$, define $\o w_s:= u_{t}^{'}- u_{s}^{'}$, we have
\begin{equation}
\begin{split}
L(\o w_s):=&\Lp (\o w_s) +K\frac{\Na u_t}{|\Na u_t|}\Na (\o w_s)\\
=&f_s\\
:=&K\left(\frac{\Na u_s}{|\Na u_s|}-\frac{\Na u_t}{|\Na u_t|}\right)\Na u_s^{'}.\\
\end{split}
\end{equation}
For all $s$, $\o w_s=0$ on $\p M$, $||\o w_s||_{L^\infty(\Me)}\leq 2$, while the equation above is with uniformly bounded coefficients. 
Therefore, $||\o w_s||_{W^{2,p}_{loc}}$ are uniformly bounded. Also note that, $\o w_s=O(|x|^{1-2q})$ and $f_s\to 0$ in $L^p_{loc}$ as $s\to t$. Then as $s \to t$,  there is a diagonal subsequence convergent to $\o w$ satisfying 
\begin{enumerate}
\item $L(\o w)=0$ in $\Me$, 
\item $\o w=0$ on $\p M$, 
\item $\o w=O(|x|^{1-2q})$ as $|x| \to \infty$.
\end{enumerate}
By maximum principle, $\o w\equiv 0$. Therefore, $c \mapsto \p_c\Psi$ is continuous. 
The same argument can be extended to multiple boundary components correspondingly.  
\end{proof}

Note that $|\Na u|=\frac{\Na u}{|\Na u|} \cdot \Na u$. Hence, the maximum principle still applies. Moreover, $u_{\vec{c}}$ is $C^2$ around $\MS$, therefore Hopf lemma also applies on each $\p_i M$. Therefore, we can follow Section 5 in \cite{HKK} to conclude Lemma \ref{ND} and Theorem \ref{top}. 

\subsection{Intersection of $\S_t$ and $\S$}\label{ITOP}
Recall from Section \ref{Preliminaries}, we have $\Me=M_{0}\cup_{i=1}^l K_i \cup_{j=1}^m \Omega_{j}.$ Notate $\cup_{i=1}^l K_i \cup_{j=1}^m \Omega_{j}$, faces and edges of $\S$ respectively by $\tilde\Omega$, $F$ and $\gamma$.  From Section \ref{Spacetime harmonic coordinates}, we know that $u|_{\tilde \Omega\setminus{\Sigma}}$ and $u|_{M_0\setminus{\Sigma}}$ are $W^{3,p}_{loc}$, $u|_{F}$ is $C^{2}$ and $u|_{\gamma}$ is $C^1$. By \cite{dP} (cf. \cite{F}), this is sufficient to conclude Sard's Theorem on these 4 functions.  Let $a$ and $b$ be the infimum and the supremum of  $u|_{\tilde\Omega}$. In particular for $u|_{\S}$, a.e. $t\in[a,b]$, $\tau_t=\{u|_{\S} = t \}$ is a closed piecewise embedded curve and since $\Sigma$ is compact, we know $\tau_t$ is of finitely many components. Since $u\in C^{1,\alpha}(\Me)$, we can see that a.e. $t\in[a,b]$, the level set $\S_t$ intersects $\S$ transversely along some closed piecewise embedded curves. 

\section{Boundary formulae} \label{Boundary formula}
With spacetime harmonic functions, we can study ADM energy and momentum by the following integral formula. 
\begin{lma}\label{IF}(cf.  \cite{T} Lemma 3.1, \cite{HKK} Proposition 3.2)
Let $(\Omega, g, k)$ be a compact initial data set with $\S:=\SS$. Then, for any spacetime harmonic function $u$ which is $C^{1,\alpha}(\bar \O)\cap C^{2,\alpha}_{loc}(\bar \Omega \setminus \bar \mE)$ , where $\mE$ denotes the edge components of $\S$, 
\be
\begin{split}
&\ \int_{\Omega} \frac12  \frac{|\o \Na \o \Na u|^2}{| \nabla u |} + \mu|\Na u| + \la J, \Na u\ra  \, d V\\
\le & \ \int_{\p_{\neq0}\Omega} \p_\nu | \nabla u | \, d \sigma +\int_{\SS} k(\Na u, \nu) d \sigma +  \frac{1}{2}\int_{\u u}^{\o u} \int_{\S_t} R_{\S_t} dA  d t, 
\end{split} 
\ee
where $\p_{\neq0}\Omega=\{x\in \po \,\ |\,\ |\Na u| \neq 0\}$, $\S_t=\{u=t\}$, $\nu$ is the outward unit normal, $\o u$ and $\u u$ denote the maximum and the minimum of $u$ respectively.
\end{lma}

\begin{proof} 
We here assume that $|\Na u|\neq 0$ for the simplicity of presentation.  For the full generality, one should first consider $\sqrt{|\Na u|^2+\delta^2}$ for $\delta>0$ and then take limit as $\d\to 0$ (see \cite{S},\cite{BS},\cite{HKK},\cite{HMT} $Remark$ 3.3).  

It suffices to verify the divergence theorem such that the following holds.  
\be\label{divergence}
\begin{split}
\int_{\p \O} \p_{\nu} |\Na u| = \int_{\O} \Lp |\Na u|. 
\end{split}
\ee

Let $\{\O_r\}_{r>0}$ be an exhaustion of $\O$ with vertices and edges of $\O$ being smoothed out, where $r$ is the parameter of radius of spherical cap around the vertices and rounded-off cylinders along the edges. The functions are regular enough on $\O_r$ so that the divergence theorem can be applied. 
\be\label{exhaustion}
\begin{split}
\int_{\p \O_r} \p_{\nu_r} |\Na u| = \int_{\O_r} \Lp |\Na u|. 
\end{split}
\ee

From a remark in the proof of Theorem 1.4 in \cite{Li0}, elliptic estimates with scaling are important in showing integrability. Let $p \in \bar \mE$,  w.l.o.g., identified as $0$ in a local coordinate chart. From the fact that $u\in C^{1,\alpha}(\bar \Omega)$ and Schauder estimates with scaling (e.g. \cite{Simon}, \cite{GT} Corollary 6.3) applied on $u$ in a (conic) annulus $A(r)$ around $p$, where $r>0$ is small, we have $|\Na \Na u|_{C^0(A(r))}\leq C r^{\alpha-1}$. Thus, $|\Na \Na u|$ is integrable on $\SS$ and $\O$. Moreover, $r|\Na\Na u| \to 0$ as $r\to 0$. Therefore, 
\be \label{boundary integral convergence}
\begin{split}
\int_{\p \O_r} \p_{\nu_r} |\Na u| \to \int_{\p \O} \p_{\nu} |\Na u|. 
\end{split}
\ee

On the other hand, $\Lp u= -K|\Na u|$, first note that by Lemma 3.1 in \cite{HKK}, we have
\be
\begin{split}
\Lp|\Na u| \geq &-C(||g||_{C^2}+||k||_{C^1})|\Na u|. 
\end{split}
\ee 
In particular,  
\be
\begin{split}
(\Lp|\Na u|)_- \leq C(||g||_{C^2}+||k||_{C^1})|\Na u|, 
\end{split}
\ee 
i.e.  $(\Lp|\Na u|)_-$ is integrable on $\O$.  

\

By \eqref{exhaustion} and integrability of $(\Lp|\Na u|)_-$,  we have on $\O_r$,
\be
\begin{split}
\int_{\O_r} (\Lp|\Na u|)_+= \int_{\p \O_r} \p_{\nu_r} |\Na u| + \int_{\O_r} (\Lp|\Na u|)_-. 
\end{split}
\ee

Hence, although $u$ is not necessarily $C^2$ on $\bar \mE$,  we can conclude that \eqref{divergence} holds by \eqref{boundary integral convergence} and monotone convergence theorem as $r\to 0$. Moreover, integrability of the integrands in \eqref{IF} follows from the elliptic estimates aforementioned.  
\end{proof}

We will express the boundary terms of Lemma \ref{IF} explicitly for spacetime harmonic functions on manifolds with boundary. 
Note that we have to make use of the fact that $\Lp u=-K|\Na u|$ instead of 0 in \cite{HMT}. 
\begin{lma}\label{BF}(cf. \cite{HMT} Proposition 2.2) 
Let $(\Omega, g, k)$ be a compact initial data set with $\S:=\SS$. Then, for any spacetime harmonic function $u$ which is $C^{1,\alpha}(\bar \O)\cap C^{2,\alpha}_{loc}(\bar \Omega \setminus \bar \mE)$ , where $\mE$ denotes the edge components of $\S$, 
\be\label{Boundary Integral}
\begin{split}
 &\int_{\S_{\neq 0}} \p_\nu | \nabla u | \, d \sigma +\int_{\S} k(\Na u, \nu) \, d \sigma\\ 
=& \ \int_{\S} \, \ \pi(\Na u,\nu)- H |\Na u| \,\  d \sigma+\int_{\u u}^{\o u} \int_{\tau_t } \kappa\, ds\,  dt\\
&+ \int_{\S_{\neq 0}}  - \frac{\nu(u)}{|\Na u|}\Delta_{_\S} \eta   +  \frac{ (\Na_\S \eta) (\nu (u))}{|\Na u|} \,\ d \sigma\\
&+ \int_{\S_{\neq 0}\cap \{ \Na_{\S} \eta\neq 0 \} }  -\frac{\nu(u)}{|\Na u|} \la \nabla_{\tau_t'} \tau_t', \Na_\S \eta \ra  \, d \sigma, 
\end{split}
\ee
where $\eta=u|_{\S}$, $\S_{\neq 0} = \{ x \in \S \,\ | \,\ |\Na u|\neq 0 \} $, $\S_t=\{u=t\}$, $\tau_t=\S_{\neq 0}  \cap \S_t\cap \{ \Na_{\S} \eta\neq 0 \} $, $H$ is computed with respect to the outward unit normal $\nu$, $\o u$ and $\u u$ are the maximum and the minimum of $u$ respectively.
\end{lma}

\begin{proof}
As discussed in Section \ref{ITOP}, a.e. $t\in[\u u, \o u]$, $t$ is regular value of $u$ and $\S_t$ intersects transversely with $\S$ on $\tau_t$ which is a closed piecewise embedded curve of finitely many components.  Then, we can consider
\be
\begin{split}
&\int_{\S_{\neq 0}} \p_\nu | \nabla u | \, d \sigma\\
= & \int_{\S_{\neq 0}} \p_\nu | \nabla u | \, d \sigma - \int_{\u u}^{\o u} \left( \int_{\tau_t } \kappa \, ds \right) dt+ \int_{\u u}^{\o u} \left( \int_{\tau_t } \kappa \, ds \right) dt .
\end{split}
\ee

We are going to to express $\p_\nu |\Na u|$ and $\kappa$ explicitly. First, for $\p_\nu |\Na u|$, we have
\be
\begin{split}
\p_\nu|\Na u|=&\frac{\Na \Na u(\Na u, \nu)}{|\Na u|}\\
=&\frac{\nu (u)}{|\Na u|}\Na \Na u (\nu, \nu)+\frac{1}{|\Na u|}\Na \Na u (\Na_\S \eta, \nu)\\
\end{split}
\ee

\vh

Using $ \Delta_\Omega u = -K |\Na u| $, we have
\be
\nabla \Na u ( \nu, \nu) = \Lp_\Omega u- H \nu(u)  - \Delta_{\S} \eta= -K|\Na u|- H \nu(u) - \Delta_{\S} \eta. 
\ee
We also have, 
\be
\begin{split}
\Na \Na u (\Na_\S \eta, \nu)=&(\Na_\S \eta)(\nu(u))-(\Na_{\Na_\S \eta} \nu)(u)\\
=&(\Na_\S \eta)(\nu(u))-\la \Na_{\Na_\S \eta} \nu, \Na_\S \eta \ra+\nu(u) \la \Na_{\Na_\S \eta} \nu, \nu \ra\\
=&(\Na_\S \eta)(\nu(u))-\Pi (\Na_\S \eta, \Na_\S \eta),\\
\end{split}
\ee
where $\Pi$ denotes the second fundamental form on $\S$ with respect to $\nu$. 

\vh

Thus, we have, 
\be\label{Boundary term general}
\begin{split}
 \p_\nu | \nabla u  |= & \ -\frac{\Pi (\Na_\S \eta, \Na_\S \eta)}{|\Na u|}    +  \frac{ (\Na_\S \eta) (\nu (u))}{|\Na u|} \\
 & -K\nu(u)- H \frac{|\nu (u)|^2}{|\Na u|} - \frac{\nu(u)}{|\Na u|}\Delta_{_\S} \eta . 
\end{split}
\ee
And in particular,  if $\Na_{\S}{\eta}=0$, 
\be\label{Boundary term constant}
\begin{split}
 \p_\nu | \nabla u  |= -K\nu(u)- H |\Na u|- \frac{\nu(u)}{|\Na u|}\Delta_{_\S} \eta . 
\end{split}
\ee

Then, we have to study the geodesic curvature $\kappa$. At a point $x \in \tau_t$, 
in particular, we know $\Na_{\S} \eta \neq 0$,
we have the following geometric vectors:

\vh

\begin{itemize}
\item $ \nu$,  the outward unit normal to $ \p \Omega$; 

\vh

\item $ \nabla_{_\S} \eta$, the gradient of $ u = \eta $ on $ \p \Omega$, which is perpendicular to $ \tau_t$ in $\S$.
We let $ n_t = \frac{1}{ | \nabla_{_\S} \eta|} \nabla_{_\S} \eta $;

\vh

\item $ \tau_t'$,  the unit tangent vector to the curve $\tau_t $; 

\vh

 \item $ \nu_{t} $, the outward unit normal to $ \tau_t $ with respect to $ \Sigma_t$; and

\vh

\item $ n = \frac{1}{|\nabla u | } \nabla u $, the normal direction to $ \Sigma_t$ along which $u$ increases. 

\end{itemize}
 Both 
$$ \{ \nu, n_t  \} \ \text{and} \ \{ \nu_t , n \} $$
are orthonormal basis for the normal bundle $ {\tau_t'}^\perp$.
We can express $n$ in the basis of $\nu$ and $n_t$, 

\be
\begin{split}
n
=&\la n,\nu \ra \nu + \la n, n_t \ra n_t\\
=&\frac{\nu (u)}{|\Na u |}\nu +\frac{|\Na_\S \eta|}{|\Na u|}n_t\\
\end{split}
\ee 
Let $ \theta \in [0, \pi]$ be the angle between $ \nu $ and $ n$, then we have 
\be \label{ct}
\begin{split}
\cos \theta = \frac{\nu (u)}{|\Na u |},\\
\end{split}
\ee
and
\be \label{st}
\begin{split}
\sin \theta = \frac{|\Na_\S \eta|}{|\Na u|}.
\end{split}
\ee

On the other hand, consider the geodesic curvature $\kappa$, by definition,
\be
\begin{split}
\kappa = & \ \la \nabla_{ \tau_t'} \nu_t , \tau_t ' \ra ,
\end{split}
\ee
and 
\be
\nu_t =  \sin \theta \, \nu - \cos \theta \, n_t. 
\ee

Thus,
\be
\begin{split}
 -\kappa = &\la \nabla_{\tau_t ' } \tau_t' , \nu_t \ra \\
= & \ \la \nabla_{\tau_t ' } \tau_t' , \,  \sin \theta \, \nu - \cos \theta \, n_t  \ra \\
= & \ - \sin \theta \, \Pi (\tau_t', \tau_t' ) - \cos \theta \, \la \nabla_{\tau_t'} \tau_t', n_t \ra \\ 
\end{split}
\ee

\bigskip

Therefore, on $\S$, 
by co-area formula, \eqref{ct} and \eqref{st},  we have, 
\be
\begin{split}
& \int_{\u u}^{\o u} \left( - \int_{\tau_t} \kappa \, d s \right) \, d t \\
=&  \int_{\S_{\neq 0}\cap \{ \Na_{\S} \eta\neq 0 \} } -|\Na_\S \eta| \sin \theta \, \Pi ( \tau_t', \tau_t') \,\ d \sigma\\ &\, 
- \int_{\S_{\neq 0}\cap \{ \Na_{\S} \eta\neq 0 \} } |\Na_\S \eta| \cos\theta \, \la \nabla_{\tau_t'} \tau_t', n_t \ra \,\ d \sigma\\
=&  \int_{\S_{\neq 0}\cap \{ \Na_{\S} \eta\neq 0 \} } - \frac{|\Na_\S \eta|^2}{|\Na u|} \,  \Pi ( \tau_t', \tau_t') \,\ d \sigma- \int_{\S_{\neq 0}\cap \{ \Na_{\S} \eta\neq 0 \} }  \frac{\nu(u)}{|\Na u|} \la \nabla_{\tau_t'} \tau_t', \Na_\S \eta \ra \,\  d \sigma.\\
\end{split}
\ee

Together with \eqref{Boundary term general}, \eqref{Boundary term constant}, we have
\be
\begin{split}
& \ \int_{\S_{\neq 0}} \p_\nu | \nabla u | \, d \sigma + \int_{\u u}^{\o u} \left( - \int_{\tau_t} \kappa \, d s \right) \, d t \\
=& \ \int_{\S_{\neq 0}\cap \{ \Na_{\S} \eta= 0 \}} \, \  -K\nu(u)- H |\Na u|- \frac{\nu(u)}{|\Na u|}\Delta_{_\S} \eta \,\ d \sigma \\
& + \ \int_{\S_{\neq 0}\cap \{ \Na_{\S} \eta\neq 0 \} } \, \ -K\nu(u)- H \frac{|\nu (u)|^2}{|\Na u|} \\
&+ \ \int_{\S_{\neq 0}\cap \{ \Na_{\S} \eta\neq 0 \} }  \,  -\frac{| \nabla_{_\S } \eta|^2}{|\Na u|} \Pi ( n_t , n_t) 
- \frac{|\Na_\S \eta|^2}{|\Na u|} \,  \Pi ( \tau_t', \tau_t') \,\ d \sigma \\
&+ \int_{\S_{\neq 0}\cap \{ \Na_{\S} \eta\neq 0 \} }  -\frac{\nu(u)}{|\Na u|} \la \nabla_{\tau_t'} \tau_t', \Na_\S \eta \ra  - \frac{\nu(u)}{|\Na u|}\Delta_{_\S} \eta   +  \frac{ (\Na_\S \eta) (\nu (u))}{|\Na u|} \,\ d \sigma  \\
=& \ \int_{\S} \, \ -Kg(\Na u,\nu)- H |\Na u| \,\  d \sigma \\
&+ \int_{\S_{\neq 0}}  - \frac{\nu(u)}{|\Na u|}\Delta_{_\S} \eta   +  \frac{ (\Na_\S \eta) (\nu (u))}{|\Na u|} \,\ d \sigma\\
&+ \int_{\S_{\neq 0}\cap \{ \Na_{\S} \eta\neq 0 \} }  -\frac{\nu(u)}{|\Na u|} \la \nabla_{\tau_t'} \tau_t', \Na_\S \eta \ra  \, d \sigma.
\end{split}
\ee

By the definition of conjugate momentum tensor $\pi$, we can conclude the lemma. 
\end{proof}

\section{Spacetime Positive Mass Theorem with Corners}\label{SPMTC}
Recall the decomposition $\Me=M_{0}\cup_{i=1}^l K_i \cup_{j=1}^m \Omega_{j}$ in Section \ref{Preliminaries}. We may first assume that $\Me=M_{0}\cup \Omega$. 
Without loss of generality, we can rotate the coordinates prior to solving for the spacetime harmonic coordinate so that $P=(-|P|, 0, 0)$.
Under the new coordinate system $x=(u,x^2,x^3)$, where $u$ is the spacetime harmonic coordinate obtained in Section \ref{Spacetime harmonic coordinates}. Let $L>>1$, define the following,
\begin{enumerate}
\item
$T_{L}=\{x\in M_0 \,\ |\,\ |u|\leq L, (x^2)^2+(x^3)^2=L^2\}$, 
\item
$D_{L}^{\pm}=\{x \in M_0 \,\  |\,\ u=\pm L, (x^2)^2+(x^3)^2\leq L^2 \}$, 
\item
$C_L=T_L \cup D_L^+ \cup D_L^-$.
\end{enumerate}
We would then label $u$ by $x^1$.  Let $M_L$ be the portion of $M_0$ bounded by $C_L $ and the corner $\S$.  Since $L>>1$, we can assume $\MS\subset M_L\cup\O$. We would use the following notations.  
\begin{itemize}
\item 
$\S_t^L=\S_t \cap M_L$,

\item 
$\S_t'=\S_t \cap \O$,
\item 
$\tau_t^L=\S_t \cap C_L$,
\item 
$\tau_t=\S_t\cap \S$,
\item 
$N$, the outward unit normal on $\p M_L$,
\item
$\nu_L$, the unit normal vector on $C_L$ pointing to the infinity of $\mathcal{E}$,
\item
 $\nu$, the unit normal vector on $\S$ pointing to the infinity of $\mathcal{E}$,
 \item
 $\nu_{\MS}$, the unit normal vector on $\MS$ pointing out of $\Me$,  
 \item
 $\MS^L$, the subcollection of $\MS$ which are in $M_0$, 
 \item
 $\MS'$, the subcollection of $\MS$ which are in $\O$,
 \item 
 $\eta=u|_{\S}$,
  \item
 $\Na^{\pm}$ and $|\cdot|_{\pm}$, the connections and norms with respect to $g_{\pm}$,
 \item
 $A_{\neq 0}=\{ x\in A \,\ | \,\ |\Na u|\neq 0 \}$ for any $A\subset M$.
\end{itemize}
From Lemma \ref{ND},  we can choose $\vec{c}$ such that $u$ on  
weakly outer trapped and weakly inner trapped components of $\MS$, we would have $\p_{\nu_{\MS}} u \leq 0$ and $\p_{\nu_{\MS}} u \geq 0$ respectively.
Furthermore, $\MS$ has empty intersection with regular level sets. 
From Section \ref{ITOP}, we know $\tau_t$ are closed piecewise embedded curves for a.e. $t$. Apply Lemma \ref{IF} on $M_L$ and Lemma \ref{BF} on $\S$, we have, 
\be\label{I1}
\begin{split}
&\ \int_{M_L} \frac12  \left( \frac{|\o \Na \o \Na u|^2}{| \nabla u |} + 2\left( \mu| \nabla u |+\la J , \Na u \ra \right) \right) \, d V\\
\le &\ \int_{\p_{\neq 0} M_L} \p_N | \nabla u | \, d \sigma +\int_{\p M_L} k(\Na u, N) \, d \sigma +  \frac{1}{2}\int_{-L}^{L} \int_{\S_t^L} R_{\S_t^L}\, dA d t\\
=&\int_{\MS^L_{\neq 0}} \p_{\nu_{\MS}} | \nabla u | \, d \sigma +\int_{\MS^L} k(\Na u, \nu_{\MS})\, d \sigma\\
&+ \int_{\p C_L} \p_{\nu_L} | \nabla u | \, d \sigma + \int_{-L}^{L} \left( - \int_{\tau_t^L} \kappa \, d s \right) \, dt+\int_{\p C_L} k(\Na u, \nu_L)\, d \sigma \\
& - \ \int_{\S} \, \ \pi_+(\Na u,\nu) - H_+ |\Na u| \,\  d \sigma \\
&+ \int_{\S_{\neq 0}}   \frac{\nu(u)}{|\Na u|}\Delta_{_\S} \eta   -  \frac{ (\Na_\S \eta) (\nu (u))}{|\Na u|} \,\ d \sigma
+ \int_{\S_{\neq 0}\cap \{ \Na_{\S} \eta\neq 0 \} } \frac{\nu(u)}{|\Na u|} \la \nabla_{\tau_t'} \tau_t', \Na_\S \eta \ra  \, d \sigma\\
&+\frac{1}{2} \int_{-L}^{L} \int_{\S_t} R_{\S_t} dA\, d t+\int_{-L}^{L} \int_{-\tau_t } \kappa\, ds\,  dt+\int_{-L}^{L} \int_{\tau^L_t} \kappa\, ds\,  dt.
\end{split}
\ee
On the other hand, by the asymptotics of $u$ and maximum principle, $\tau_t^L$ is a circle. By computations in Section 6 of \cite{BKKS} and Section 6 of \cite{HKK}, we get, 
\be\label{I2}
\begin{split}
&\ \int_{\p C_L} \p_{\nu_L} | \nabla u | \, d \sigma + \int_{-L}^{L} \left( - \int_{\tau_t^L} \kappa \, d s \right) \, dt+\int_{\p C_L} k(\Na u, \nu_L)\, d \sigma \\
=&-4\pi L+\frac{1}{2}\int_{C_L} \left( g_{ij,i}-g_{ii,j} \right) \nu_L^j \, dA+ \int_{C_L} \pi_{1j} \nu_L^j \, dA+ O(L^{1-2q})+O(L^{-q}). 
\end{split}
\ee

Similarly, apply Lemma \ref{IF} on $\O$ and Lemma \ref{BF} on $\S$, we have, 
\be \label{I3}
\begin{split}
&\ \int_{\O} \frac12  \left( \frac{|\o \Na \o \Na u|^2}{| \nabla u |} + 2\left( \mu| \nabla u |+\la J , \Na u \ra \right) \right) \, d V\\
\leq &\ \int_{\MS'_{\neq 0}} \p_{\nu_{\MS}} | \nabla u | \, d \sigma +\int_{\MS'} k(\Na u, \nu_{\MS})\, d \sigma + 2\pi  \int_{-L}^{L}  \chi (\Sigma_t' )\, d t\\
&+ \ \int_{\S} \, \ \pi_-(\Na u,\nu) - H_- |\Na u| \,\  d \sigma \\
&+ \int_{\S_{\neq 0}}  - \frac{\nu(u)}{|\Na u|}\Delta_{_\S} \eta   +  \frac{ (\Na_\S \eta) (\nu (u))}{|\Na u|} \,\ d \sigma \\
&+ \int_{\S_{\neq 0}\cap \{ \Na_{\S} \eta\neq 0 \} }  -\frac{\nu(u)}{|\Na u|} \la \nabla_{\tau_t'} \tau_t', \Na_\S \eta \ra  \, d \sigma\\
&+\frac{1}{2} \int_{-L}^{L} \int_{\S_t} R_{\S_t} dA\, d t+\int_{-L}^{L} \int_{\tau_t } \kappa\, ds\,  dt..
\end{split}
\ee
By \eqref{Boundary term constant}, on $\MS$, with the corresponding choice of sign of normal derivatives as aforementioned,
\be\label{I4}
\begin{split}
&\ \int_{\MS_{\neq 0}} \p_{\nu_{\MS}} | \nabla u | + k(\Na u, \nu_{\MS})\, d \sigma\\
=& \sum_{i=1}^n \int_{\p_iM_{\neq 0}}  H |\p_{\nu_{\MS}} u|-tr_{\p_iM}k(\p_{\nu_{\MS}} u) \, d\sigma \\
\leq &\, 0,
\end{split}
\ee 
where $\p_i M$ are the components of $\MS$ and $H$ is computed with respect to $-\nu_{\MS}$.

\

Note that $u$ is $C^1$ across $\S$, $\nu(u)+(-\nu)(u)$ is constantly zero on $\S$.  Furthermore,  $g$ is continuous, in case $\pm \tau_t$ have some turning angles, they are of the opposite signs. Moreover, by Theorem \ref{top}, we know that $\S_t$ has a single end modeled on $\R^2\setminus B_1$. Therefore, for $L>>1$, a.e. $t\in [a,b]$, $1\geq \chi(\S_t)=\chi(\S_t^L)+\chi(\S_t')$. Summing equations \eqref{I1} and \eqref{I3}, and applying Gauss-Bonnet Theorem, we have, 
\be
\begin{split}
&\ \int_{M_L\cup \O} \frac12  \left( \frac{|\o \Na \o \Na u|^2}{| \nabla u |} + 2\left( \mu| \nabla u |+\la J , \Na u \ra \right) \right) \, d V\\
\leq & \ 2\pi  \int_{-L}^{L}  \chi (\Sigma_t )\, dt -4\pi L+\frac{1}{2}\int_{C_L} \left( g_{ij,i}-g_{ii,j} \right) \nu_L^j \, dA+ \int_{C_L} \pi_{1j} \nu_L^j \, dA\\
&- \ \int_{\S} \, \ \pi_+(\Na u,\nu) - H_+ |\Na u| \,\  d \sigma + \ \int_{\S} \, \ \pi_-(\Na u,\nu) - H_- |\Na u| \,\  d \sigma \\
&+ O(L^{1-2q})+O(L^{-q})\\
\leq &\, \frac{1}{2}\int_{C_L} \left( g_{ij,i}-g_{ii,j} \right) \nu_L^j \, dA+ \int_{C_L} \pi_{1j} \nu_L^j \, dA\\
&+ \ \int_{\S} \, \left( H_+-H_- \right) |\Na u| \,\  d \sigma + \ \int_{\S} \,  (\pi_--\pi_+)(\Na u,\nu) \,\  d \sigma \\
& + O(L^{1-2q})+O(L^{-q}). 
\end{split}
\ee
By Proposition 4.1 in \cite{Bartnik}, as $L\to \infty$, we have
$$\frac{1}{2}\int_{C_L} \left( g_{ij,i}-g_{ii,j} \right) \nu_L^j \, dA+ \int_{C_L} \pi_{1j} \nu_L^j \, dA \to 8\pi \left( E+P_1 \right) = 8\pi \left( E-|P| \right).$$ 

For the general case, we can apply the same idea onto each component $M_0$, $K_i$ and $\Omega_j$ and sum up the integrals. Theorem \ref{main} is therefore proved. 

As we can see from the proof above, particularly the term $\pi_{+}(\Na u, \nu)$, in general we get
\begin{cor}\label{gen}For $\va\in \Sp\subset \R^3$, if a spacetime harmonic function $\ua$ is asymptotic to $a^ix^i$, then  
\be
\begin{split}
16\pi (E+\la \va,P \ra)\ge & \int_{M_{ext}\setminus \S}\left(\frac{|\o \Na \o \Na \ua|^2}{|\nabla \ua|}+2(\mu|\nabla \ua|+ \la J, \Na \ua \ra) \right)\\
&+2\int_\S(H_--H_+)|\nabla \ua|-2\int_\S (\pi_--\pi_+)(\Na \ua, \nu).
\end{split}
\ee
\end{cor}

\section{Proof of Corollary \ref{rigid}}\label{Embedding into Minkowski space}
We can follow Section 7 in \cite{HKK} with slight modifications to conclude the equality case. Some details are provided to explain how to deal with discontinuity of $k$ and $u$ being only $C^{1,\alpha}$ across $\S$. 
\subsection{$E=|P|$ case} 
Under the assumptions of Corollary \ref{rigid}, from the inequality of Theorem \ref{main}, if $E=|P|$, we have $\o \Na \o \Na u=\Na \Na u + |\Na u|k=0$.  Then, by Kato's inequality, we get 
$$|\Na |\Na u||\leq |\Na \Na u|\leq |k||\Na u|.$$
Therefore, by standard ODE technique (see Lemma \ref{l-independent-2} in the next section or Lemma 7.1 in \cite{HKK}), there exists a constant $C>0$ such that $|\Na u|\geq C$ on $M_0$. Since $u\in C^{1,\alpha}_{loc}$, $|\Na u|\geq C$ on $\S$.  Hence, use the same technique again within the remaining compact portions of $\Me$, we can conclude that $|\Na u|\geq \tilde C>0$ for some $\tilde C$ on $\Me$. This is inconsistent with the choice of normal derivatives of $u_{\vec{c}}$ on $\MS$ by Lemma \ref{ND}.  Hence, $\MS$ is empty. 

\

Let $\tau>>1$, $\S_\tau=\{u=\tau\}$ is an asymptotically flat complete plane and since $|\Na u|$ does not vanish, we can see along the level set flow, the topology does not change. Therefore, $M$ is diffeomorphic to $\R^2\times \R =\R^3$.  
Moreover, we can see that on each level set $\S_t$, $\frac{\Na \Na u}{|\Na u|}|_{T\S_t}+k|_{T\S_t}=h_t+k|_{T\S_t}=0$ (spacetime totally geodesic), where $h_t$ is the second fundamental form of $\S_t$ with respect to $\frac{\Na u}{|\Na u|}$.  $M$ is thus foliated by stable MOTS.  Then, by Theorem 1 (2) of \cite{Carlotto}, we know that each $\S_t$ has vanishing Gauss curvature and hence is isometric to $\R^2$.  Thus, together with asymptotic flatness, the metric can be expressed as $g(u,x^2,x^3)=\frac{1}{|\Na u|^2}du^2+\delta_{ij}dx^idx^j$.  

\subsection{Isometric embedding into Minkowski space for the case $E=|P|=0$.} 
As $\MS$ is empty, we now have $M=M_{0}\cup_{i=1}^l K_i$. Here, for notation simplicity, we denote $\cup_{i=1}^l K_i$ by $\tilde{K}$.  

For $M\cong \R^3$,  let \,\ $(x^1,x^2,x^3)$ be a global coordinate system which coincides with the asymptotically flat coordinate on $M\setminus\mathcal{C}$. And we can,  as in Section \ref{Spacetime harmonic coordinates}, construct a spacetime harmonic function $u(a_1,a_2,a_3)$ which is asymptotic to $a_i x^i$, where $\sum_{i=1}^3 (a_i)^2=1$. Then as in Theorem 7.3 in \cite{HKK}, we can define a lapse function $\alpha$ and a shift vector $\beta$ by
\be
\alpha=\left| \Na u \left( \frac{1}{\sqrt 2}, \frac{1}{\sqrt 2}, 0 \right) \right|+\left| \Na u \left( -\frac{1}{\sqrt 2}, 0,  \frac{1}{\sqrt 2} \right) \right|-\left| \Na u \left( 0, \frac{1}{\sqrt 2}, \frac{1}{\sqrt 2} \right) \right|,
\ee 
and
\be
\beta=\Na u \left( \frac{1}{\sqrt 2}, \frac{1}{\sqrt 2}, 0 \right) +\Na u \left( -\frac{1}{\sqrt 2}, 0, \frac{1}{\sqrt 2} \right) -\Na u \left( 0, \frac{1}{\sqrt 2}, \frac{1}{\sqrt 2} \right).
\ee
Then, we can define a stationary spacetime, $(\o M = \R \times M,  \o g)$ where 
\be
\o g=-(\alpha^2-|\beta|^2)dt^2 + 2\beta_i dx^idt +g,
\ee
where the Killing vector is 
\be
\p_t=\alpha \vec{n}+\beta,
\ee
where $\vec{n}$ is the unit normal to the hypersurface constant $t$-slice. We can see that $(M,g)$ is isometric to a constant time slice in $\o M$ under such construction. First, notice that $\alpha$ and $\beta$ are differentiable on $M\setminus \S$ and continuous across $\S$. From equations (7.9) to (7.11) in \cite{HKK}, it is shown that $\alpha^2-|\beta|^2$ is constant in $M\setminus \tilde{K}$ and $\tilde{K}\setminus\Sigma$ respectively. By continuity, we have $\alpha^2-|\beta|^2$ is a constant on $M$. Since $\alpha^2-|\beta|^2 \to 1$ as $r \to \infty$, we have  $\alpha^2-|\beta|^2\equiv 1$. we thus have, 
\be
\begin{split}
\o g=-dt^2 + 2\beta_i dx^idt +g=&-(dt-\beta_i dx^i)^2+(g_{ij}+\beta_i \beta_j)dx^idx^j\\
=&-(dt-d\Psi)^2+(g+d\Psi)^2, 
\end{split}
\ee 
where $\Psi= u \left( \frac{1}{\sqrt 2}, \frac{1}{\sqrt 2}, 0 \right) +u \left( -\frac{1}{\sqrt 2}, 0, \frac{1}{\sqrt 2} \right) - u \left( 0, \frac{1}{\sqrt 2}, \frac{1}{\sqrt 2} \right)$. Notice that $\beta$ is exact since $\beta=\Na \Psi$. Then, on $M\setminus \S$, we have, $a, b, c = 0, 1, 2, 3$, where $\p_0=\p_t$. 
$$\o\Gamma_{it}^{a}=\frac{1}{2}\o g^{ac}\left( \o g_{ic,t}+\o g_{tc,i}-\o g_{it,c} \right)=\o g^{ac}\left( \p_i \beta_c-\p_c \beta_i \right)=0.$$
On the other hand, since $E=|P|=0$, for all $|\vec{a}|=1$, we have $\Na \Na u (a_1, a_2, a_3)=-|\Na u (a_1,a_2,a_3)|k$, and hence
$$\Na_i \beta_j=-\alpha k_{ij}.$$

With these, we can show $k$ is the corresponding 2nd fundamental from of $M$ with respect to this embedding since on $M\setminus \Sigma$, 
$$\la \o \Na_i \vec{n}, \p_j \ra= \alpha^{-1}\la \o \Na_i(\p_t-\beta), \p_j\ra=\alpha^{-1}\o\Gamma_{it}^b \o g_{bj}+\alpha^{-1} \Na_i \beta_j=k_{ij}.$$ 
Therefore, $(M,g,k)$ arises as a constant time slice in $(\o M, \o g)$. 

\

For $l=1,\, 2$ and $3$, construct vector fields $X_l$ on $M$ as follows, 
$$X_l=\Na u_l+|\Na u_l|\vec{n},$$
where $$u_1=u(1,0,0), \,\ u_2=u(0,1,0), \,\  u_3=u(0,0,1),$$ i.e. the spacetime harmonic coordinates corresponding to the original asymptotically flat coordinates $(x^1, x^2, x^3)$. These vector fields are differentiable on $M\setminus \S$ and continuous across $\S$. 
Extend these vector fields trivially along $\p_t$ to $\o M$. Then, as shown in equations (7.13), (7.18) to (7.21) in \cite{HKK}, we know that on $\R\times (M\setminus \tilde{K})$ and $\R\times (\tilde{K}\setminus\Sigma)$, these vector fields and $\p_t$ are covariantly constant. And hence by continuity, the metric on these vector fields is constant on $\o M$.  
They are linearly independent at the asymptotic end and thus linearly independent on $\o M$.  Therefore, $(\o M, \o g)$ is flat. Further, by a change of coordinate, $\o t=t-\Psi(x)$ and $\o x=x$, we have 
$$\o g=-d\o t^2+(g+d\Psi^2).$$ 
From this construction, we can see $(M,g,k)$ can be expressed as a graph $\o t =-\Psi(\o x)$. Also note that $(\R^3, g+d\Psi^2)$ is asymptotically flat and therefore complete. Furthermore, it is a constant $\o t$ slice in this splitting of $\o g$ and hence it is flat and isometric to Euclidean space. Therefore, we have $\o M$ is isometric to Minkowski space. 

\section{Proof of Corollary \ref{regids}}\label{Raising Regularity}
In this section, we assume $\S$ is smooth. As in Corollary \ref{gen}, for $\va \in \Sp$, let $u_{\va}$ denote the spacetime harmonic function asymptotic to $a^ix^i$. Under the assumptions, and by Corollary \ref{rigid}, if $E=|P|$, then there exists a spacetime harmonic function $u=u_{\frac{-P}{|P|}}$ such that 
\be \label{tensorwise}
\Na \Na u= - |\Na u| k \text{\, on } M\setminus\S , 
\ee

\be 
|\Na u|\geq c>0  \text{\, on } M,
\ee
and 
\begin{equation}
\begin{split}
0\ge &\int_{M\setminus\S}\left(\frac{|\o \Na \o \Na u|^2}{|\nabla u|}+2(\mu|\nabla u|+ \la J, \Na u \ra) \right)\\&+2\int_\S(H_--H_+)|\nabla u|-2\int_\S (\pi_--\pi_+)(\Na u, \nu)\\
\ge &\int_{M\setminus\S}\left(\frac{|\o \Na \o \Na u|^2}{|\nabla u|}+2(\mu-|J|_g)|\nabla u| \right)\\&+2\int_\S \left( H_--H_+-|\omega_--\omega_+| \right ) |\nabla u| \\
\geq &0.\\
\end{split}
\end{equation} 
We thus have on $M\setminus\S$, 
\be
\begin{split}
\mu|\nabla u|+ \la J, \Na u \ra=\mu |\nabla u|-|J|_g ||\nabla u|=0,
\end{split}
\ee
and on $\Sigma$,
\be
\begin{split}
&(H_--H_+)|\nabla u|-2 (\pi_--\pi_+)(\Na u, \nu)\\
=&( H_--H_+-|\omega_--\omega_+| ) |\nabla u|\\
=&0.
\end{split}
\ee
Therefore, on $M\setminus\S$, 
\be\label{J}
J=-|J|_g\frac{\nabla u}{|\nabla u|}=-\mu\frac{\nabla u}{|\nabla u|}
\ee and on $\S$, 
\be\label{pi}
Z=-|Z|\frac{\nabla u}{|\nabla u|}=-(H_--H_+)\frac{\nabla u}{|\nabla u|},
\ee 
where $Z$ denotes the vector field dual to $(\pi_--\pi_+)(\cdot, \nu)$. 

\

\subsection{$E=|P|=0$}\label{EP0}
We first would see what  $E=|P|=0$ can imply. Note that if $E=|P|=0$, the relations mentioned above are not only satisfied by $u_{\frac{-P}{|P|}}$ but also by $\ua$ for all $\va\in \Sp$. 

In particular due to the non-vanishing gradient, we can define a vector field on $M$ by, 
$$
X_{\mathbf{a}}(p)=\frac{\nabla u_{\mathbf{a}}}{|\nabla u_{\mathbf{a}}|}(p)
$$
for $\va\in \Sp$, $p\in M$. Then the vector field is continuous on $M$, $C^{1,\alpha}$ outside $\Sigma$ and
$$
X_{\mathbf{a}}(p)\to  \mathbf{a}
$$
as $p\to\infty$. Fix $p\in M$, we can define a map
$$
F_p:\Sp\to \Sp(p)
$$ by
$$
F_p(\mathbf{a})=X_{\mathbf{a}}(p), 
$$
where $\Sp(p)$ is the unit sphere in $T_p(M)$.

Consider the following two key lemmas. 
 \begin{lma}\label{l-independent-1}
Let $X=\nabla u/|\nabla u|$  and let   $Y=\nabla \wt u/|\nabla \wt u|$ where $u$ and $\wt{u}$ are  spacetime harmonic functions, then on $M\setminus\Sigma$,
$$
|\nabla (|X-Y|^2)|\le 2|k||X-Y|^2.
$$

\end{lma}
\begin{proof}
\be 
\begin{split}
\Na X=&\Na (\frac{\Na u}{|\Na u|})\\
=&\frac{\Na \Na u}{|\Na u|}-\frac{1}{|\Na u|^2}\frac{\Na \Na u( \Na u, \cdot)}{|\Na u|}\Na u\\
=&-k+k(X,\cdot)X,
\end{split}
\ee
i.e. in local coordinates, $\Na_i X^j=-k_i^j+k_{mi}X^mX^j$.
Similarly, 
\be 
\begin{split}
\Na Y=-k+k(Y,\cdot)Y.
\end{split}
\ee
Hence, 
\be 
\begin{split}
\Na (|X-Y|^2)=&\Na (|X|^2+|Y|^2-2\la X, Y\ra) \\
=&-2\la \Na X, Y \ra - 2\la X, \Na Y \ra  \\
=&2(k(Y,\cdot)-k(X,\cdot)\la X, Y\ra+k(X,\cdot)-k(Y,\cdot)\la X, Y\ra)\\
=&2(1-\la X, Y \ra)k(X+Y,\cdot)\\
=&|X-Y|^2k(X+Y,\cdot).
\end{split}
\ee
And
\be 
\begin{split}
|\Na (|X-Y|^2)|
\leq &|X-Y|^2 |k|(|X|+|Y|)\\
\leq& 2|X-Y|^2||k|.
\end{split}
\ee
\end{proof}

\begin{lma}\label{l-independent-2}
If $E=|P|=0$, $F_p$ is a homeomorphism for all $p\in M$.
\end{lma}
\begin{proof} By the invariance of domain, it suffices to show that $F_p$ is injective and continuous. Recall the notation of decomposition of $M$ as in the proof of Corollary \ref{rigid}, $M=M_0\cup \tilde{K}$. Let $p\in M_0\setminus\Sigma$ and $\gamma:[0,\infty)\to M_0\setminus\Sigma$ be a smooth curve so that
$\gamma(0)=p$ and $\gamma$ is a straight line near infinity. For any $\mathbf{a}, \mathbf{b}\in \Sp$. By Lemma \ref{l-independent-1}, outside $\Sigma$,
$$
|\nabla (|X_\mathbf{a}-X_\mathbf{b}|^2)|\le 2|k||X_\mathbf{a}-X_\mathbf{b}|^2.
$$
Define
$$
f(t)=|X_\mathbf{a}-X_\mathbf{b}|^2(\gamma(t)),
$$
then by the decay rate of $k$,
$$
|f'(t)|\le C_1|k|(\gamma(t))f(t)\le C_2 (1+t)^{-1-q} f(t).
$$
for some constants $C_1, C_2$ independent of $\mathbf{a}, \mathbf{b}$ and $t$. Hence we have for all $t$,
$$
f(0)e^{ -\int_0^t C_2 (1+s)^{-1-q}ds}\le f(t)\le f(0)e^{ \int_0^t C_2 (1+s)^{-1-q}ds}.
$$
Let $t\to\infty$, we have
$$
|X_\mathbf{a}-X_\mathbf{b}|^2(p)e^{-\int_0^\infty C_2 (1+s)^{-1-q}ds}\le |\mathbf{a}-\mathbf{b}|\le |X_\mathbf{a}-X_\mathbf{b}|^2(p)e^{ \int_0^\infty C_2 (1+s)^{-1-q}ds}.
$$
Since the integral
$$\int_0^\infty (1+s)^{-1-q} ds $$
converges. By continuity of $X_{\va}-X_{\mathbf{b}}$, the above also holds on $\S$. Since $\tilde{K}$ is compact, we can apply the same argument. Therefore, the lemma follows.

\end{proof}

\begin{cor}\label{c-independent-1}
If $E=|P|=0$, then for any $p\in M$, there exist $\mathbf{a_1}, \mathbf{a_2}, \mathbf{a_3}$ in $\Sp$ such that
$X_{\mathbf{a_1}}, X_{\mathbf{a_2}}, X_{\mathbf{a_3}}$ are linearly independent at $p$. In fact, they can be chosen to be orthonormal.
\end{cor}

\begin{cor}\label{muJpiHk}
If $E=|P|=0$, then $\mu=|J|=0$ on $M\setminus\Sigma$, $\pi_-(\cdot,\nu)=\pi_+(\cdot,\nu)$, $H_-=H_+$ and $tr_{\S}k_-=tr_{\S}k_+$ on $\Sigma$. 
\end{cor}
\begin{proof}
By Corollary \ref{c-independent-1}, \eqref{J} and \eqref{pi}, we can see that $J=-|J|_g \vec{v}$ for all $\vec{v}\in \Sp(p)$ for all $p \in M\setminus\Sigma$ and $Z=-|Z|\vec{w}$ on $\Sigma$ for all $\vec{w}\in\Sp(q)$ for all $q \in \Sigma$. And by considering $tr_{\S}k_{\pm}=\pi_{\pm}(\nu,\nu)$, the result thus follows. 
\end{proof} 
\subsection{$E=|P|=0$ and conditions on $k$}
After the preparation in previous sections, we are now ready to show that the regularity of $(g,k)$ can be improved by further assuming $k$ is continuous and the normal derivative of $tr_g k$ is continuous on $\Sigma$. 
 
\

We would notate a spacetime harmonic function by $u$ and by $\ua$ if its dependence on $\va\in\Sp$ has to be emphasised. As discussed in Section \ref{Preliminaries}, we can assume near $\Sigma$, $M$ is of the form $\Sigma\times(-\e,\e)$. Hence $g$ can be written in the form
$$
g(x,t)=g_+(x,t)=dt^2+h_+(x,t) \text{\, if } t\geq0;
$$
$$
g(x,t)= g_-(x,t)=dt^2+h_-(x,t) \text{\, if } t\leq0,
$$
where $h_\pm$ is the metric induced by $g$ on $\Sigma\times \{t\}$. In the following, $ A, B$ etc. are indices for tensors on $\Sigma$, $1\le A, B\le 2$. While $\p_t$ is denoted by the index 3 and $i, j$ etc. are from 1 to 3.

{\bf 1.} \underline{$g$ is $C^{1,1}$}: 
To prove $g$ is $C^1$, it suffices to show the second fundamental forms from both sides match. It was shown in Section \ref{Spacetime harmonic coordinates} that $u$ is $C^{2}$ away from $\Sigma$ and $u|_{\Sigma}$ is $C^2$ on $\Sigma$. Therefore, from both sides, $u$ is $C^2$ up to $\Sigma$. By \eqref{tensorwise} and continuity from both sides. We have on $\Sigma$, 
$$\Na^-_A\Na^-_Bu=(k_{-})_{AB}|\Na u| \text{\, and\, } \Na^+_A\Na^+_Bu=(k_{+})_{AB}|\Na u|.$$ 
Since $k_-$ and $k_+$ induce the same tensor on $\Sigma$,   
$$\Na^-_A\Na^-_Bu=\Na^+_A\Na^+_Bu.$$ 
Then we can follow the argument as in \cite{HMT} Section 3. Observe that
\be
\nas_A \nas_B u + II_+(\partial_A, \partial_B)\partial_{\nu}u=\nas_A \nas_B u + II_-(\partial_A, \partial_B)\partial_{\nu}u, 
\ee
where $II_{\pm}$ stands for the second fundamental form for $g_{\pm}$ with respect to $\partial_t$. Now,  since $g_{\pm}$ induce the same smooth metric on $\Sigma$,  
\be
II_+(\partial_A, \partial_B)\partial_{t}u= II_-(\partial_A, \partial_B)\partial_{t}u.
\ee 
By Corollary \ref{c-independent-1}, for each $p\in \Sigma$, we can find $\ua$ such that $\partial_{t}\ua(p)$ is non zero. We therefore can conclude that the second fundamental forms match on $\Sigma$. We can see that $g$ is $C^1$. Since the second derivatives of $g_\pm$ are bounded, so $g$ is $C^{1,1}$.

{\bf 2.} \underline{$\ua$ is $C^{2,\a}$ for all $\mathbf{a}$}. Now $g$ is $C^{1,1}$ since $k$ is assumed to be Lipschitz.  Hence,  $\Delta \ua=-(\tr_g k)|\Na \ua|$ implies $\ua\in C^{2,\alpha}_{loc}(\Me)$ by standard elliptic regularity theory.  

{\bf 3.} \underline{$k$ is $C^{1,1}$ under the assumption that $tr_g k$ is $C^1$} Now, we are also going to make use of the assumption that the normal derivative of $tr_g k$ is continuous. Together with $k$ being continuous, we can say $tr_g k$ is $C^1$. 
From this onward, $F_{,i}$ means partial derivative and $F_{;i}$ means covariant derivative.
We first prove that $\pi$ is $C^1$. Since on $\Sigma$, $\pi^+=\pi^-$, we have
$$
\pi^+_{ij,A}=\pi^-_{ij,A}.
$$

At a point on $\Sigma$, w.l.o.g., assume $\p_1, \p_2$ are orthonormal. By Corollary \ref{muJpiHk} we have for each $i$, 
$$
J_i=0=\pi_{ij;j}^\pm =\pi_{i1;1}^\pm+\pi_{i2,2}^\pm+\pi_{i3;3}^\pm
$$
Now
$$
\pi_{i1;1}^+=\pi_{i1,1}^+-\Gamma_{1i}^k\pi_{k1}^+-\Gamma_{11}^k\pi_{ik}^+,
$$
$$
\pi_{i1;1}^-=\pi_{i1,1}^--\gamma_{1i}^k\pi_{k1}^--\gamma_{11}^k\pi_{ik}^-,
$$
where $\Gamma, \gamma$ are connections of $g_+, g_-$ respectively. Since $g$ is $C^1$, we have on $\Sigma$, 
$$
\pi_{i1;1}^+=\pi_{i1;1}^-.
$$
Similarly, $$
\pi_{i2;2}^+=\pi_{i2;2}^-
$$
and so

$$
\pi_{i3;3}^+=\pi_{i3;3}^-
$$
on $\Sigma$. This implies that on $\Sigma$
$$
\pi_{i3,3}^+=\pi_{i3,3}^-. 
$$

It remains to prove that
$$
\pi_{AB,3}^+=\pi_{AB,3}^-.
$$
On $\Sigma$, since 
$
u_{;AB}=-k_{AB}|\nabla u|,
$
we have
\be
\begin{split}
u_{;A3B}^+=&\p_{B}(u_{;A3}^+)-\Gamma_{BA}^ku_{k3}^+-\Gamma_{B3}^ku_{Ak}^+\\
=&\p_{B}(k^+_{A3}|\Na u|)-\Gamma_{BA}^ku_{k3}^+-\Gamma_{B3}^ku_{Ak}^+.
\end{split}
\ee
\be
\begin{split}
u_{;A3B}^-=&\p_{B}(u_{;A3}^-)-\gamma_{BA}^ku_{k3}^+-\gamma_{B3}^ku_{Ak}^-\\
=&\p_{B}(k^-_{A3}|\Na u|)-\gamma_{BA}^ku_{k3}^+-\gamma_{B3}^ku_{Ak}^-.
\end{split}
\ee
Since $g$ is $C^1$, $u$ is $C^2$ and $k$ is continuous, we have on $\Sigma$, 
\be\label{u3coder}
u_{;A3B}^+=u_{;A3B}^-.
\ee
Moreover, 
$$
u_{;AB3}^\pm=-k_{AB}^\pm (|\nabla u|)_{;3}-|\nabla u|k_{AB;3}^\pm.
$$

Now consider, 
\bee
\begin{split}
|\nabla u|k_{AB;3}^+-|\nabla u|k_{AB;3}^-=u_{;AB3}^--u_{;AB3}^+.
\end{split}
\eee

We claim that for each $p\in \S$, one can find $\mathbf{a}\in \Sp$ such that the right hand side is zero for  $u=u_{\mathbf{a}}$. If the claim is true, then $k_{AB,3}^+=k_{AB,3}^-$ by the fact that $\Na u$ is nowhere vanishing. To prove the claim, by the Ricci identity, 
\bee
\left\{
  \begin{array}{ll}
   u_{;AB3}^+=&u_{;A3B}^++R_{iA3B}^+u^i; \\
   u_{;AB3}^-=& u_{;A3B}^-+R_{iA3B}^-u^i
  \end{array}
\right.
\eee
By Lemma \ref{l-independent-2}, we can find $\mathbf{a}$ so that $X_{\mathbf{a}}=\p_A$. For this particular $u$, we have
$$
u_{;AB3}^+= u_{;A3B}^+;\  u_{;AB3}^-= u_{;A3B}^-.
$$
Hence, the claim is true by \eqref{u3coder}, $k_{AB;3}^+=k_{AB;3}^-$ and $k_{AB,3}^+=k_{AB,3}^-$.
To summarise, we have shown that
\bee
\left\{
  \begin{array}{ll}
    \pi_{ij,A}^+=\pi_{ij,A}^- \\
     \pi_{i3,3}^+=\pi_{i3,3}^-\\
     k_{AB,3}^+=k_{AB,3}^-.
  \end{array}
\right.
\eee
 Since $\pi=k-(\tr_g k)g$, we have
  $$
 \pi_{12;3}^\pm=k_{12;3}^\pm
 $$
 $$
 \pi_{AA;3}^\pm=k_{AA;3}^\pm-(\tr_g k_\pm)_{;3}
 $$

Now by the assumption that normal derivative of $tr_g k$ is continuous,  we know that $\pi_{AA;3}^+=\pi_{AA;3}^-$. As $g$ is $C^1$, we know that $\pi$ is $C^1$ hence $C^{1,1}$ as the connection matches. Therefore, $k$ is also $C^{1,1}$.

 {\bf 4.} \underline{Curvature tensor is continuous}
 Let $\mathbf{a}\in \Sp$ and denote $X=X_{\mathbf{a}}$ as in Section \ref{EP0}. We thus have
 
  \bee
  \begin{split}
  X^i_{;j}= -g^{is}k_{sj}+X^iX^sk_{sj}.
  \end{split}
  \eee
  and so
  $$
  X_{i;j}=-k_{ij}+X_iX^sk_{sj}.
  $$
  Hence
  \bee
  \begin{split}
  X_{i;jm}=(-k_{ij}+X_{i}X^sk_{sj})_{;m}
  \end{split}
  \eee
  Since $u$ is $C^2$, $k$ is $C^1$ and $g$ is $C^1$, so $X_{i;jm}$ is $C^0$. On the other hand, by the Ricci identity
  $$
  X_{i;jm}-X_{i;mj}=-R^l_{ijm}X_l=g^{lp}R_{pijm}X_l=-R_{pijm}X^p.
  $$
Hence the right hand side is also continuous. On the other hand, by Corollary \ref{c-independent-1}, locally we can always find linearly independent continuous vector fields $X, \wt X, \ol X$ so that the above are true. Hence for each $i,j,m$, $R_{pijm}$ is in $C^0$. Hence $Rm$ is continuous.
  
  {\bf 5.} \underline{$g$ is $C^{2,1}$} As in Lemma 4.1 in \cite{ST}, it remains to check that
  $$
  \frac{\p^2}{\p t^2}h_{AB}
  $$
  is continuous. Now 
  \bee
  \begin{split}
   \frac{\p^2}{\p t^2}h_{AB}=&-2\frac{\p}{\p t}\la \p_t,\nabla_{\p_A}\p_B\ra\\
   =&-2\la \nabla_{\p_t} \p_t ,\nabla_{\p_A}\p_B\ra-2\la  \p_t ,\nabla_{ \p_t}\nabla_{\p_A}\p_B\ra\\
   =&-2\la  \p_t ,\nabla_{ \p_t}\nabla_{\p_A}\p_B\ra  \text{\ \ (as $\nabla_{\p_t}\p_t=0$)}\\
   =&-2\la \p_t,\nabla_{\p_A}\nabla_{\p_t}\p_B\ra-2\la\p_t,R(\p_t,\p_A)\p_B\ra\\
   =&-2\p_A\la \p_t,\nabla_{\p_t}\p_B\ra+2 \la \nabla_{\p_A}\p_t, \nabla_{\p_t}\p_B\ra
   -2\la\p_t,R(\p_t,\p_A)\p_B\ra\\
   =&2 \la \nabla_{\p_A}\p_t, \nabla_{\p_t}\p_B\ra
   -2\la\p_t,R(\p_t,\p_A)\p_B\ra \text{\ \ (as $|\p_t|=1$)}.
  \end{split}
  \eee
The first term is the square of the second fundamental form while curvature tensor was shown to be continuous above. This proves that $g$ is $C^2$, hence $C^{2,1}$. Moreover, $u$ is $C^{3,\alpha}_{loc}$ by elliptic regularity (\cite{GT} Theorem 9.19) since $g\in C^{2,1}_{loc}$, $k\in C^{1,1}_{loc}$ and $|\Na u|\in C^{1,\alpha}_{loc}$.  Corollary \ref{regids} is therefore proved.  

\

On the other hand, by \cite{Carlotto}, we can conclude the following. 
\begin{prop}
Let $(M,g,k)$ be an initial data set as in Theorem \ref{main}.  Assume that the dominant energy condition holds on $\Me\setminus \S$ and $$(H_--H_+)-|\omega_--\omega_+|\geq0$$ on $\S$. 
Further assume $(g,k)$ has (boosted) harmonic asymptotics (\cite{Carlotto} Definition 2.3, \cite{EHLS} Definition 4), then $E=|P|$ implies $E=|P|=0$. 
\end{prop}
\begin{proof}
From the analysis above, we know that $M$ is foliated by MOTS if $E=|P|$.  In particular, each level set of the spacetime harmonic function $u$ asymptotic to $\frac{-|P|}{|P|}$ is a stable MOTS on which $h_t+k|_{T\S_t}=0$ (spacetime totally geodesic) by \eqref{tensorwise}, where $h_t$ is the second fundamental form of $\S_t$ with respect to $\frac{\Na u}{|\Na u|}$.  

Then, note that for $t>>1$, $\S_t$ is a $C^3$ stable MOTS. If  $(g,k)$ has (boosted) harmonic asymptotics, then we can apply \cite{Carlotto} Theorem 2 to conclude that $E=|P|=0$. 
\end{proof}

\section{Connections to quasilocal mass}\label{Connections to quasilocal mass}

\subsection{Positivity of $\MW(\S)$}
In \cite{ST}, the Riemannian positive mass theorem with Lipschitz metric along corners is used to prove positivity of the Brown York mass. This motivates us to consider if Theorem \ref{main} can provide an insight into some quasilocal quantities. 
In particular, we have the following. 
\begin{cor}\label{W}
Let $(\Omega^3, g, k)$ be a compact initial data set satisfying the dominant energy condition. Assume there exists $\MS$, a finite (possibly empty) disjoint union of connected weakly trapped surfaces, such that $H_2(\Omega_{ext}, \MS, \mathbb{Z})=0$, where $\Omega_{ext}$ denotes the portion of $\Omega$ outside $\MS$. Suppose $\S=\SS$ is a smooth surface of finitely many components with Gaussian curvature $\kappa>0$ and mean curvature $H$ with respect to the outward normal $\nu$. Denote the mean curvature of isometric embedding of $\S$ into $\R^3$ with respect to the outward normal by $H_0$. If $H > |\omega|$, where $\omega=\pi(\cdot,\nu)$, then 
$$\mathcal{W}(\S):=\frac{1}{8\pi}\int_\S H_0 - \left( H-|\omega| \right) \geq 0.\footnote{One can compare $\MW(\S)$ to the expression of the physical Hamiltonian in equation (2.14) in \cite{HH}. }$$ 
If $\MW(\S)=0$, then $\S$ is connected, $\Omega$ is diffeomorphic to a domain in $\R^3$ and can be isometrically embedded into Minkowski space. 
\end{cor}
\begin{proof}
We would follow Bartnik-Shi-Tam construction of quasi-spherical metric (\cite{B}, \cite{ST}) for each component $\S_i$ of $\S$. First,  define on $\S$, 
\be
u=\frac{H_0}{H-|\omega|}.
\ee
As Nirenberg (\cite{Nirenberg}) and independently, Pogorelov (\cite{Pogorelov}) have solved Weyl’s isometric embedding problem, we know that by its positive Gauss curvature, $\S_i$ can be isometrically embedded into $\R^3$. 
We notate the image of isometric embedding of $\S_i$ into $\R^3$ by $\S_{0i}$, and the unbounded region of $\R^3$ outside of $\S_{0i}$ by $M_{i}=\S_{0i} \times [0,\infty)$ which stands for a foliation by unit normal flow. 
Then as in \cite{ST}, we can construct an asymptotically flat metric $g_{i}=u(r)^2dr^2+g_r$ with zero scalar curvature on $M_{i}$ (\cite{ST} Theorem 2.1(b)), where $u(0)=u$ and $g_r$ stands for the metric induced on $\S_r=\Sigma_{0i} \times \{t=r\}$ by the Euclidean metric on $\R^3$. Since the Gauss curvature of $\S_r$ is positive, we have (Lemma 4.2 in \cite{ST}), 
\be
\begin{split}
8\pi \frac{d}{dr}Q(\S_r)
:=&\frac{d}{dr} \int_{\S_r}H_0(r)\left( 1-\frac{1}{u(r)} \right)d\sigma_r \\
=&-\frac{1}{2} \int_{\S_r} R_{\S_r} u^{-1}(1-u)^2\leq 0, 
\end{split}
\ee
where $H_0(r)$ is the mean curvature of $\S_r$ with respect to the Euclidean metric of $\R^3$. Moreover, by Theorem 2.1 (c) in \cite{ST}, we have$$\lim_{r\to \infty}Q(\S_r)= E(g_{i}).$$ Therefore to prove the inequality, i.e. $\MW(\S_i)=Q(\S_{0i})\geq 0$, it suffices to show $E(g_{i})\geq 0$. 

\

Consider the glued initial data set $\tilde M=\Omega \cup_{\Sigma_i, i=1}^n M_{+i}$, with metric $\tilde g=(g, \{g_{i} \} _{i=1}^n)$ and symmetric 2 tensor $\tilde k=(k, \{ k_{i}=0\}_{i=1}^n)$, and correspondingly $\pi_{i}=0$ for $i=1,2,...,n$. By the construction above, we know that $\tilde g$ is Lipschitz across $\S$ and the dominant energy condition is satisfied on $\tilde M \setminus \Sigma$. And by equation (1.6) in \cite{ST}, we have on $\Sigma_i$, the mean curvature $H_i$ of $g_i$ with respect to the outward normal is $H-|\omega|$. Therefore, on $\Sigma_i$, we have 
$$H-H_i-|\omega-0|=H-(H-|\omega|)-|\omega|=0.$$ 
Fix $l$, for $g_{l}$, as discussed in Section \ref{Preliminaries}, for each of other extensions, a large coordinate sphere can act as a weakly trapped surface with respect to $g_{l}$. These spheres together with $\MS$ are the boundary of $\tilde M_{ext}$. Hence, by Theorem \ref{main}, we have $0\leq E(g_{l})-|P_{l}|=E(g_{l})$. As $\MW(\S)=\sum_{i=1}^n\MW(\S_i)$, we can conclude the positivity. 
And by Corollary \ref{rigid}, we can conclude the equality case.  
\end{proof}

As we can see from the proof based on \cite{ST} above, if we consider spin condition as in \cite{Shibuya} and \cite{LL}, we can arrive at the following conclusion. 
\begin{cor}(cf. \cite{ST} Theorem 4.1)
For $n\geq 3$,  let $(\Omega^n, g, k)$ be a compact initial data set in a spacetime $N^{n+1}$ satisfying the dominant energy condition.  Assume that $\S:=\SS$ has finitely many components. Let $H$ denote the mean curvature of $\S$ with respect to the outward normal $\nu$.  Suppose $\Omega$ is spin and $\S$ can be isometrically embedded into $\R^{n}$ as a strictly convex closed hypersurface.  Denote the mean curvature of isometric embedding of $\S$ into $\R^n$ with respect to the outward normal by $H_0$. If $H > |\omega|$, where $\omega=\pi(\cdot,\nu)$, then 
$$\mathcal{W}(\S):=\frac{1}{8\pi}\int_\S H_0 - \left( H-|\omega| \right) \geq 0,$$ 
and equality implies that $\Sigma$ is connected and $N$ is a flat spacetime along $\Omega$.  
\end{cor}

\subsection{Comparison to the Liu-Yau mass}
In \cite{LY1}, \cite{LY2} and \cite{K}, a general version of the Brown-York mass, known as the Liu-Yau mass, is proposed for the spacetime case. 
\begin{thm}(\cite{LY1}, \cite{LY2})
Let $(\Omega^3, g, k)$ be a compact initial data set in a spacetime $N$ satisfying the dominant energy condition. Suppose $\S:= \SS$ is a smooth surface of finitely many components with Gaussian curvature $\kappa>0$ and mean curvature $H$ with respect to the outward normal. Denote the mean curvature of isometric embedding of $\S$ into $\R^3$ by $H_0$. If $H > |tr_\S k|$,
then 
$$m_{LY}(\S):=\frac{1}{8\pi}\int_\S H_0 - \sqrt{H^2-|tr_\S k|^2}\geq 0,$$
and equality implies that $\Sigma$ is connected and $N$ is a flat spacetime along $\Omega$.  
\end{thm}

Let $A$ denote the index for tensor of  $T\Sigma$. Since 
$$\pi(\nu,\nu)=k(\nu,\nu)-(tr_gk)g(\nu,\nu)=-tr_{\Sigma}k,$$ 
$$\pi(\p_A,\nu)=k(\p_A,\nu)-(tr_gk)g(\p_A,\nu)=k(\p_A,\nu),$$ 
we have 
$$|\omega|^2=|\pi(\cdot,\nu)|^2=|\pi(\nu,\nu)|^2+|\pi(\cdot,\nu)^T|^2=|tr_{\Sigma}k|^2+|k(\cdot,\nu)^T|^2,$$ 
where $^T$ denotes the restriction of a tensor onto $T\S$.  

If $H>|tr_{\S}k|$,  then pointwisely on $\Sigma$, we have
\be
\begin{split}
&H_0-(H-|\omega|)\\
\geq & H_0-(H-|tr_{\Sigma} k|)\\
\geq & H_0-(\sqrt{H+|tr_{\Sigma} k|}\sqrt{H-|tr_{\Sigma} k|})\\
= & H_0-(\sqrt{H^2-|tr_{\Sigma} k|^2}).\\
\end{split}
\ee 
Upon integration on $\Sigma$,  we thus have $$\MW(\Sigma)\geq m_{LY}(\Sigma).$$ 
Therefore, we have the following result.  
\begin{prop}
Let $(\Omega^3, g, k)$ be a compact initial data set, which is a compact spacelike hypersurface in a spacetime $N$ satisfying the dominant energy condition. Suppose $\S:= \SS$ is a smooth surface of finitely many components with Gaussian curvature $\kappa>0$ and mean curvature $H$ with respect to the outward normal. Denote the mean curvature of isometric embedding of $\S$ into $\R^3$ by $H_0$. If $H > |tr_\S k|$,
then 
$$\mathcal{W}(\S):=\frac{1}{8\pi}\int_\S H_0 - \left( H-|\omega| \right) \geq 0,$$  
where $\omega=\pi(\cdot,\nu)$.
Moreover,  equality implies that $\Sigma$ is connected and $N$ is a flat spacetime along $\Omega$.  
\end{prop}

\begin{rem}
On the other hand, it is worthwhile to mention that the proof of Corollary \ref{W} by Theorem \ref{main} does not need to consider the Jang graph as in \cite{LY1, LY2}.  This property gives an extra application of $\MW(\S)$ which will be discussed in Section \ref{Applications of W}. 
\end{rem}

\subsection{Rigidity of $\MW(\S)$ in $\R^{3,1}$}
From the proof of positivity above, we cannot conclude whether $\MW(\S)=0$ if $\Omega$ lies in a flat spacetime. 
Following the idea of \cite{MST}, we can see the rigidity of $\MW(\S)$ if $\Omega$ is in $\R^{3,1}$.

\begin{cor}\label{RM}(cf. \cite{MST} Theorem 4.1)
Let $(\Omega,^3 g, k)\subset \R^{3,1}$ be a compact initial data set.  If  $\S=\p \O$ is smooth and connected, has positive Gaussian curvature and $H>|tr_{\S}k|$, where $H$ is the mean curvature of $\S$ with respect to the outward normal, then $\MW(\S)\geq 0$.  Moreover, $\MW(\S) = 0$ if and only if $\O$ is a domain within a hyperplane in $\R^{3,1}$. 
\end{cor}
\begin{proof}
Since the mean curvature vector of $\Omega$ is spacelike.  Then we can follow the proof of Theorem 4.1 in \cite{MST} to construct a maximal spacelike slice $(\tilde\Omega, \tilde g, \tilde k)$ spanned by $\Sigma$ such that $\tilde{H}>0$, where $\tilde{H}$ is the mean curvature of $\S$ with respect to the outward normal of $\tilde \Omega$. 

\
 
On $\tilde \Omega$, since $tr_{\tilde g} \tilde k=0$, we have $0=2\tilde \mu=\tilde R-|\tilde k|_{\tilde g}^2$, in particular, $\tilde R\geq 0$. We then consider
\be
\begin{split}
&8\pi\MW(\S)\\
\geq &8\pi m_{LY}(\S)\\
= & \int_{\S}H_0-|\vec{H}| \\
\geq &\int_\S H_0 - \tilde H, 
\end{split}
\ee
where $\vec{H}$ is the mean curvature vector for $\Sigma \subset \R^{3,1}$ and $H_0$ is the mean curvature of isometric embedding of $\S$ into $\R^3$. Note that, since we have positive scalar curvature and $\tilde H>0$, we can conclude $\MW(\S)\geq 0$ by directly applying Theorem 1 of \cite{ST}. Hence, we do not need extra topological assumptions.

For the second part, if $\MW(\S)=0$, then we have $(\Omega,g)$ must be a flat domain by \cite{ST}, in particular, we have $R=0$. Therefore, $k\equiv0$ and $\O$ is a domain within a hyperplane in $\R^{3,1}$. On the other hand, if $\O$ is a domain within a hyperplane in $\R^{3,1}$, then we know $\MW(\S)=0$ by definition.  
\end{proof}

\section{Applications of $\MW(\S)$}\label{Applications of W}

\subsection{Comparison Theorem}
In \cite{ST3} and \cite{ALY}, it is shown that comparisons of different quasilocal masses which depend on the boundary data yield conclusions on the internal geometry of the initial data sets.  Following their ideas, and by Theoerm \ref{main}, merely on the initial data set level, we first have obtained a comparison theorem of $\MW$ and the Hawking mass on admissible initial data sets, which is defined as follows. 

\begin{defn}\label{admIDS} (cf.  \cite{ALY} Definition 1.1, 2.3)
An initial data set $(\O^3,g,k)$ is admissible if 
\begin{enumerate}
\item $\O$ is compact and simply connected, 
\item there exists $\MS$, a finite (possible empty) disjoint union of connected weakly trapped surfaces, such that $H_2(\Omega_{ext}, \MS, \mathbb{Z})=0$, where $\Omega_{ext}$ denotes the portion of $\Omega$ outside $\MS$, 
\item the dominant energy condition holds, i.e. $\mu\geq|J|_g$ on $\Omega$, 
\item $\S:=\SS$ is smooth, connected and has a positive Gaussian curvature, 
\item on $\S$, $H>|\omega|\not\equiv 0$, where $H$ is computed with respect to the outward normal $\nu$ and $ \omega=\pi(\cdot,\nu)$, and
\item  there exists $d_1>0$ depending on the Sobolev constant of an asymptotically flat extension $(\tilde M= \O \cup_{\S} M_+, \tilde{g}=(g,g_+))$ and $d_2>0$ depending on $||\omega||_{L^1(\S)}$ such that 
\be\label{eqK}
\begin{split}
||K^2||_{L^{\frac{3}{2}}(\O)}<d_1,\,\,\,\,\,\,\,\,\,\,\,\ ||K^2||_{L^{\frac{6}{5}}(\O)}<d_2,
\end{split}
\ee
where $K=tr_g k$ and $d_1$, $d_2$ would be determined in the proof of Proposition \ref{conf}.
\end{enumerate}
\end{defn}

Let us recall the ideas of minimising hulls (\cite{HI}) and the Hawking mass.
\begin{defn}
Let $E$ be a set in $\O$ with locally finite perimeter. $E$ is said to be a minimizing hull in $\O$  if $|\p^*E \cap W|\leq|\p^*F \cap W|$ for any set $F$ with locally finite perimeter
such that $F \supset E$ and $F \setminus E \subset \subset \O $, and for any compact set $W$ with $F\setminus E \subset W \subset \O$.
Here $\p^*E$  and  $\p^*F$ are the reduced boundaries of $E$ and $F$ respectively. $E$ is said to be
a strictly minimizing hull if equality (for all $W$) implies $E \cap \O  = F \cap \O $  a.e.
\end{defn}

Suppose $E$ is an open set of $\O$ such that there is a strictly minimizing hull in $\O$ containing $E$, then define $E'$ to be the intersection of all strictly minimizing hulls containing $E$. $E'$ is called the strictly minimizing hull of $E$. As pointed out in \cite{ST3}, we have the following. Let $E \subset \subset \O$ be an open set. Suppose $\O$ is compact with smooth boundary which has positive mean curvature with respect to the outward normal, then $E'$ exists and $E'\subset\subset\O$ . 

\begin{defn}
Let $E$ be an open set in a Riemannian manifold with compact $C^1$ boundary, the
Hawking mass of $\p E$ is defined as 
\be
m_H(\p E)=\sqrt{\frac{|\p E|}{16\pi}}(1-\frac{1}{16\pi}\int_{\p E}{H^2}), 
\ee
where $|\p E|$ is the area of $\p E$.
\end{defn}

In this section the ADM energy would be denoted by $m_{ADM}$. Now, we can state the comparison theorem. 

\begin{thm}\label{PI}(cf. \cite{ST3} Theorem 3.1, \cite{ALY} Theorem 1.2)
Let $(\Omega^3,g,k)$ be admissible. Then, for any connected minimising hull $E$ in $\Omega$ where $E\subset \subset \Omega$ with $C^{1,1}$ boundary $\p E$, we have 
$$\MW(\S)\geq m_H(\p E).$$
\end{thm}
\begin{proof}
First of all, we would follow Section 3 of \cite{ALY} to construct an asymptotically flat extension of $(\O,g,k)$. Since $H>|\omega|\geq0$, we can construct the Shi-Tam extension $(M_+,g_+,k_+=0)$ by $u=\frac{H_0}{H}$ instead of the function used in the proof of Corollary \ref{W}. Then, by \cite{ST}, we know that the Brown-York mass satisfies 
\be\label{BYmonotone}
m_{BY}(\Sigma):=\frac{1}{8\pi}\int_{\S} \left( H_0-H \right) dA_g\geq m_{ADM} (g_+),
\ee
where $H_0$ is the mean curvature of isometric embedding of $\S$ into $\R^3$ with respect to the outward normal. On $\S$, $H_+=H$, where $H_+$ is the mean curvature of $g_+$ on $\S$ with respect to the normal pointing to $M_+$. We also have $R_{g_+}=0$.

Then, for the glued initial data set $(\ti M, \ti g, \ti k)=(\O \cup_{\S}M_+,(g,g_+),(k,0))$, $\ti g$ is Lipschitz across $\S$. We would mollify the initial data set along the signed distance direction of $\S$ as in Section 3 of \cite{M1} (cf.  Section 5 of \cite{AKS}). In particular, we have the following. 
\begin{lma}\label{MollifiedIDS}
There exists a family of smooth initial data $\{(g_{\d}, k_{\d})\}_{\d>0}$ with the following properties: 
\begin{enumerate}
    \item $( g_\d, k_\d)=(\ti g, \ti k)$ outside $\S\times(-\d,\d)$, 
    \item $||g_\d||_{C^{0,1}(\ti M)}$ is uniformly bounded, $|| k_\d||_{L^{\infty}(\ti M)}$ is uniformly bounded, 
    \item $R_{\delta}\geq-C$  inside $\S\times(-\d,\d)$, where $C>0$ independent of $\d$. 
\end{enumerate}
\end{lma}
Let $B_{\d}=2\mu_\d=R_{\d}+K_{\d}^2-|k_{\d}|_{\d}^2$. 
Note that $R_g+K^2-|k|_g^2=2\mu\geq2|J|_g\geq 0$ and  $R_{g_+}+K_+^2-|k_+|_{g_+}^2=0$, in particular, we know that 
\be 
\begin{split}
B_{\d-}&=0 \,\ \,\  \,\ \,\ \text{outside} \,\ \S \times(-\d,\d) \text{ and} \\
B_{\d-}&=O(1) \,\ \text{inside} \,\ \S \times(-\d,\d), 
\end{split}
\ee
where $B_{\d-}:=\max\{0, -\mu_{\d}\}$ and $O(1)$ is a bounded quantity independent of $\d$.  

\

Next, we are going to seek a conformal factor to make $\ti M$ have positive scalar curvature. To be precise, we are going to show the following.
\begin{lma}\label{conf}
Given the deformation in Lemma \ref{MollifiedIDS}, and let $F \subset\subset \O \subset \ti M$, for sufficiently small $d_1>0$ for \eqref{eqK} depending on the Sobolev constant of asymptotically flat extension manifold $(\ti M, \ti g)$, there exist a $C^2(\ti M\setminus F)$ positive function $u_\d\geq 1$ such that the conformal metric $\hat g_{\d}=u_{\d}^4 g_{\d}$ has non-negative scalar curvature on $\ti M\setminus F$.  Moreover, there exists $d_2>0$ depending on $||\omega||_{L^1(\S)}$ for \eqref{eqK} such that 
$$\MW(\S)\geq \o \lim_{\d\to 0}\, m_{ADM}(\hat g_{\d}).$$
\end{lma}
\begin{proof}
Let $b_{\d}=B_{\d-}+K_{\d}^2$, which is compactly supported on $\O_{\d}:=\O \cup (\S\times [0,\d])$. Then, let us consider the following PDE, 
\begin{enumerate}
\item $\Lp_{\d} u_{\d}+\frac{1}{8}b_{\d}u_{\d}=0$  on $\ti M\setminus F$,
\item $u_{\d}\to 1$ as $|x|\to \infty$,
\item $\p_{\nu}u_{\d}=0$ on $\p F$.
\end{enumerate}
Let $w_{\d}=u_{\d}-1$, we have 
\begin{enumerate}
\item $\Lp_{\d} w_{\d}+\frac{1}{8}b_{\d}w_{\d}=-\frac{1}{8}b_{\d}$  on $\ti M\setminus F$,
\item $w_{\d}\to 0$ as $|x|\to \infty$,
\item $\p_{\nu}w_{\d}=0$ on $\p F$.
\end{enumerate}
Following Lemma 3.2 in \cite{SY1}, there exists $d_0>0$ which depends on the Sobolev constant of $g_{\delta}$ such that if $||b_{\d}||_{L^{\frac{3}{2}}(\ti M \setminus F, g_{\d})}<d_0$,  then a positive solution exists by Fredholm alternative.  Since $||B_{\delta-}||_{L^{\frac{3}{2}}(\ti M \setminus F,  g_{\d})} \to 0$ as $\d\to 0$, the existence of solution is guaranteed. Moreover,  by Lemma \ref{MollifiedIDS}, $d_0$ can be chosen independent of $\d$.  While  it can further be shown that $w_{\d}\in C^2(\ti M \setminus F)$ with the asymptote $w_{\d}=\frac{A_{\d}}{|x|}+O^2_{\d}(|x|^{-2})$ if there exists a bound on $||w_{\d}||_{L^6(\ti M\setminus F)}$. Multiply $w_\d$  to the PDE and integrate by parts, apply Holder inequality and Young inequality, together with Sobolev inequality, we get

\be
\begin{split}\label{L6Estimate}
&\frac{8}{C_{s}}\left( \int_{\MF} w_{\d}^6 \, dV_{\d} \right)^{\frac{1}{3}}\\
\leq & 8 \int_{\MF} |\Na^{\d}w_{\d}|_{\d}^2 \, dV_{\d}\\
\leq & \left( \int_{\MF} b_{\d}^{\frac{3}{2}} \, dV_{\d} \right)^{\frac{2}{3}} \left( \int_{\MF} w_{\d}^6 \, dV_{\d}\right)^{\frac{1}{3}} + \left(\int_{\MF} b_{\d}^{\frac{6}{5}} \, dV_{\d} \right)^{\frac{5}{6}} \left( \int_{\MF} w_{\d}^6 \, dV_{\d}\right)^{\frac{1}{6}} \\
\leq & \left(\int_{\MF} B_{\d-}^{\frac{3}{2}} \, dV_{\d} \right)^{\frac{2}{3}} \left( \int_{\MF} w_{\d}^6 \, dV_{\d}\right)^{\frac{1}{3}}
+\left( \int_{\MF} K_{\d}^{\frac{6}{2}} \, dV_{\d} \right)^{\frac{2}{3}} \left( \int_{\MF} w_{\d}^6 \, dV_{\d}\right)^{\frac{1}{3}} \\
&+ \left(\int_{\MF} B_{\d-}^{\frac{6}{5}} \, dV_{\d} \right)^{\frac{5}{6}} \left( \int_{\MF} w_{\d}^6 \, dV_{\d}\right)^{\frac{1}{6}}
+ \left(\int_{\MF} K_{\d}^{\frac{12}{5}} \, dV_{\d} \right)^{\frac{5}{6}} \left ( \int_{\MF} w_{\d}^6 \, dV_{\d}\right)^{\frac{1}{6}} \\
\leq & \, C(\d|\S|)^{\frac{2}{3}} \left( \int_{\MF} w_{\d}^6 \, dV_{\d}\right)^{\frac{1}{3}}
+d_1 \left( \int_{\MF} w_{\d}^6 \, dV_{\d}\right)^{\frac{1}{3}} \\
&+\frac{C(\d|\S|)^{\frac{5}{3}}}{4Q}\, +Q\left( \int_{\MF} w_{\d}^6 \, dV_{\d}\right)^{\frac{1}{3}}
+\frac{d_2^2}{4Q}+ Q\left( \int_{\MF} w_{\d}^6 \, dV_{\d}\right)^{\frac{1}{3}},
\end{split}
\ee
where $C_{s}>0$ is the Sobolev constant which can be chosen independent of $\d$, so are $d_1$, $d_2$ by Lemma \ref{MollifiedIDS}, while $Q>0$ is to be determined.   A bound on $L^6$ norm of $w_\d$ follows if
$$\frac{1}{8}\left( C(\d|\S|)^{\frac{2}{3}}+d_1 +2Q \right)\leq \frac{1}{2C_{s}}.$$ As only the case $\d$ being small is concerned, we can pick $Q$ sufficiently small to ensure that there exists $d_1>0$ satisfying the inequality.  

We can then verify the positivity of $R_{\hat \d}$, 
\be
\begin{split}
R_{\hat \d}=&\, u_{\d}^{-5}\left( R_{\d}u_{\d}-8\Lp_{\d}u_{\d} \right)\\
=&\, u_{\d}^{-5}\left( (B_{\d}+|k_{\d}|_{\d}^2-K_{\d}^2)u_{\d} + (B_{\d-}+K_{\d}^2)u_{\d} \right) \\
\geq &\, u_{\d}^{-5}\left( |k_{\d}|_{\d}^2u_{\d} \right) \\
\geq & 0.\\
\end{split}
\ee

We can determine $d_2$ now.  By the asymptotics of $u_{\d}$, we have, 
\be\label{massforlimit}
m_{ADM}(\hat g_{\delta})=m_{ADM}(g_{\delta})+2A_{\delta}=m_{ADM}(g_+)+2A_{\delta}, 
\ee
and 
\be
\begin{split}
0\leq A_{\d}=&\frac{1}{32\pi}\int_{\MF}b_{\delta}u_{\d}\, dV_{\delta} \\
\leq & \frac{C(\d|\S|)^{\frac{5}{6}}}{32\pi}  \left( \int_{\O_{\d}\setminus F} u_{\d}^6 \, dV_{\d}\right)^{\frac{1}{6}}
+  \frac{d_2}{32\pi}  \left ( \int_{\O_{\d}\setminus F} u_{\d}^6 \, dV_{\d}\right)^{\frac{1}{6}}.
\end{split}
\ee
By \eqref{L6Estimate} and Lemma \ref{MollifiedIDS}, we can see for all sufficiently small $\d$,  \\
$||w_{\d}||_{L^6(\MF, g_{\d})}$ is uniformly bounded by a term involving $d_2$. 
Hence $||u_{\d}||_{L^6(\O_{\d}\setminus F, g_{\d})}$ is uniformly bounded.  Altogether, $A_{\d}\leq C(d_2+\delta)$ for some $C>0$. Therefore, we can choose $d_2>0$ such that for all small $\d$, $A_{\delta}$ is smaller than some constant $C(|\omega|)>0$
which is to be determined.  

Now, by \eqref{BYmonotone}, 
\be
\begin{split}
\MW(\S)=&\frac{1}{8\pi }\int_\S H_0-(H-|\omega|) \, dA_g\\
=&\frac{1}{8\pi }\int_\S \left(  H_0-H \right) \, dA_g + \frac{1}{8\pi }\int_\S |\omega| \, dA_g\\
\geq & m_{ADM}(g_+) + \frac{1}{8\pi }\int_\S |\omega| \, dA_g. 
\end{split}
\ee
If we choose sufficiently small $d_2>0$ such that $\o\lim_{\d\to0}\, A_{\d} \leq C(|\omega|):= \frac{1}{16\pi }\int_\S |\omega| \, dA_g$. 
Then, by boundedness, we can take $\o \lim_{\d \to0}$ on equation \eqref{massforlimit}. To sum up, we have 
\be
\begin{split}
\MW(\S)\geq \o \lim_{\d\to 0}\, m_{ADM}(\hat g_{\d}).\\
\end{split}
\ee
\end{proof} 

Finally, to prove Theorem \ref{PI}, it suffices to show $\o \lim_{\d\to 0}\,m_{ADM}(\hat g_{\d})\geq m_H(\p E)$. 
We can follow the argument of Section 3 of \cite{ST3}. Let $E$ be a connected minimising hull in $\O$ where $E\subset \subset \O$ with $C^{1,1}$ boundary. As $(\Omega, g, k)$ is admissible, by Corollary \ref{qm}, $\MW(\S)\geq0$. Hence, it suffices to prove for the case $m_H(\p E)>0$, in particular, we have $|\p E|_g>0$. 

\

Let $\theta>0$, we can find a connected set $F\supset E$ with smooth boundary $\p F$ such that $F\subset\subset \O$ satisfies the following.  
\begin{enumerate}
\item $|\p E|_g-\theta\leq |\p F|_g \leq |\p E|_g+\theta$, 
\item $m_H(\p E)\leq m_H(\p F)+\theta$,
\item $m_H(\p F)>0$.
\end{enumerate}

Let $\eps>0$, by Lemma \ref{MollifiedIDS}, 
there exists $\d_0\in(0, \eps)$ such that for all $0<\d<\d_0$, 
\be \label{metc}
(1-\eps) \ti g \leq g_{\d} \leq(1+\eps) \ti g.
\ee

Let $\d\in (0,\d_0)$. By the properties of the Bartnik-Shi-Tam extension constructed, as pointed out by \cite{ST3} item (vii)  in their proof of Theorem 3.1, and $||w_{\delta}||_{C^1}\to 0$ as $|x|\to\infty$, we have the existence of $F'$ which is the strictly minimising hull of $F$ in $(\ti M, \hat g_{\d})$. Moreover, $F'$ is precompact  in $\ti M$ and has connected $C^{1,1}$ boundary $\p F'$. By equation (3.9) in \cite{ST3}, we have 
\be \label{area}
|\p F'|_{\ti g} >|\p E|_g. 
\ee 

Then, 
\be
\begin{split}
&m_{ADM}(\hat g_{\d})\\ 
\geq &\sqrt{\frac{|\p F'|_{\hat \d}}{16\pi}} \left( 1-\frac{1}{16\pi}\int_{\p F'} H_{\hat \d}^2 dA_{\hat \d} \right)\\
\geq & \sqrt{\frac{|\p F'|_{\hat \d}}{|\p F|_{g}}} \sqrt{\frac{|\p F|_{g}}{16\pi}} \left( 1-\frac{1}{16\pi}\int_{\p F} H_{\hat \d}^2 dA_{\hat \d}  \right)\\
= & \sqrt{\frac{|\p F'|_{\hat \d}}{|\p F|_{g}}} \sqrt{\frac{|\p F|_{g}}{16\pi}} \left( 1-\frac{1}{16\pi}\int_{\p F} H_{g}^2 dA_{g}  \right)\\
= & \sqrt{\frac{|\p F'|_{\hat \d}}{|\p F|_{g}}} m_H(\p F)\\
\geq & \sqrt{\frac{|\p F'|_{\hat \d}}{|\p F|_{g}}} m_H(\p E)-\sqrt{\frac{|\p F'|_{\hat \d}}{|\p F|_{g}}} \theta \\
\geq & \sqrt{\frac{|\p F'|_{\d}}{|\p F|_{g}}} m_H(\p E)-\sqrt{\frac{|\p F|_{\hat \d}}{|\p F|_{g}}} \theta \\
\geq & \sqrt{\frac{(1-\eps)|\p F'|_{\ti g}}{|\p F|_{g}}}m_H(\p E)- \max_{\p F}|u_{\d}^2|  \theta \\
\geq & \sqrt{\frac{(1-\eps)|\p E|_{g}}{|\p E|_{g}+\theta}}m_H(\p E)- \max_{\Omega}|u_{\d}^2| \theta.
\end{split}
\ee
The second line follows from the Geroch monotonicity of Hawking mass along inverse mean curvature flow (\cite{HI} Formula 5.8 and Lemma 7.4, \cite{Geroch}). The third line holds because $H_{\hat \d}=0$ on $\p F'\setminus \p F$ and the mean curvatures of $\p F'$ and $\p F$ agrees a.e. on $\p F' \cap \p F$. The fourth line comes from the fact that conformal changes of the mean curvature and the volume form cancel each other due to the Neumann conditions of the conformal factor. The seventh line is due to $u_{\d}\geq 1$ and by the choice of $F'$. The remaining follows from \eqref{metc},  \eqref{area} and Lemma \ref{MollifiedIDS}. 
Let $\theta\to 0$, and then $\eps\to 0$, i.e. $\d\to0$, the result follows.
\end{proof}
\begin{rem}
If $\O$ is a maximal slice, then as shown in the proof of Corollary \ref{RM}, the dominant energy condition implies positive scalar curvature. Hence, the topological assumption item (2) can be omitted. 
\end{rem}
As observed from the proof above, we have actually shown a more general statement.  
\begin{cor} 
Let $(\O^3,g,k)$ satisfy items (1) to (4), and have $d_1>0$ in item (6) of the admissibility condition. If $H>0$, then for all $\eps>0$, there exists $d_2=d_2(\eps)>0$ such that if $||K^2||_{L^{\frac{6}{5}}(\O)}<d_2$, then
\be
m_{BY}(\S)+\eps \geq m_H(\p E).
\ee
\end{cor}

As a corollary of Theorem \ref{PI}, we have the following. 
\begin{cor}\label{Pen}
Let $(\O^3,g,k)$ be admissible. Suppose that $S$ is a outward minimising surface in $\Omega$, then 
$$\MW(\S)\geq\sqrt{\frac{|S|}{16\pi}}.$$
\end{cor}
\begin{proof}
Since an outward minimising surface is a boundary of a connected minimising hull.  Then by Theorem \ref{PI}, we have, 
$$\MW(\S)\geq m_H(S)=\sqrt{\frac{|S|}{16\pi}}.$$
\end{proof}

\subsection{Detection of minimal surfaces}
In this section, following the approach by \cite{ST3} and \cite{ALY}, we are going to see how $\MW(\S)$, the Hawking mass 
and the Shi-Tam mass help study the internal geometry of an admissible initial data set, in terms of existence and non-existence of minimal surfaces. 
\begin{prop}(cf. \cite{APM} Theorem 2.3, \cite{ALY} Proposition 1.7) Let $(\Omega,g,k)$ be admissible. Suppose that the sectional curvatures of $\Omega$ are bounded above by some constant $C^2$, where $C>0$. If $\MW(\S)<\frac{1}{2C}$, then there is no minimal surface in $\Omega$ and $\Omega$ is diffeomorphic to a ball in $\R^3$. 
\end{prop}
\begin{proof}
The proof follows from \cite{Cor} Theorem 2.3. Assume $\Omega$ is not diffeomorphic to a ball in $\R^3$. Since $\S$ is mean convex, by Theorem 1' in \cite{MSY}, there exist an outward minimizing minimal sphere $S$ in $\O$. For a given point $p\in S$, let$ \{ e_1, e_2 \}$ be an orthonormal basis of $T_pS$ on which the 2nd fundamental form of $S$ is diagonal.  Denote
the principal curvatures at $p$ by $\kappa_i$. The Gauss equation implies 
\be
\frac{R_{S}}{2}=Rm_{S}(e_1,e_2,e_2,e_1)=Rm_{\Omega}(e_1,e_2,e_2,e_1)+ \kappa_1 \kappa_2.
\ee
Since $S$ is minimal, $\kappa_1 \kappa_2\leq 0$ at $p$ and therefore we have
\be
\frac{R_S}{2}\leq C^2.
\ee
By Gauss–Bonnet Theorem and the assumption, we have
\be
4\pi= \int_{S} \frac{R_S}{2} \leq |S|C^2 < \frac{|S|}{4\MW(\S)^2}.
\ee
Together with Corollary \ref{Pen}, we have, 
\be
\MW(\S)\geq\sqrt{\frac{|S|}{16\pi}}>\MW(\S),
\ee
which is a contradiction. Therefore,  $\O$ has no minimal surfaces and by \cite{MSY}, $\Omega$ is diffeomorphic to a ball in $\R^3$. 
\end{proof}

Recall the following definition and lemma (\cite{ST3} Lemma 3.6). 
\begin{defn}
A $C^2$ surface $V\subset (\O,g)$ is called an isoperimetric surface if its area is no more than any other
$C^2$ surface enclosing the same volume.
\end{defn}

\begin{lma}\label{MMSMH}
Let $(\O, g)$ be a compact Riemannian 3-manifold with smooth mean convex boundary $\p \O$ such that Gaussian curvature of $\p \O$ is positive. 
Suppose $E\subset \subset \Omega$ such that $\p E$ is a isoperimetric surface. Then either $\O$ contains a outward minimising minimal surface or $E$ is a minimising hull. 
\end{lma}

From the lemma, we have the following. 
\begin{prop}
Let $(\Omega,g,k)$ be admissible. If there is an
isoperimetric surface $V\subset \Omega$ with $m_H(V)> \MW(\S)$, then there is a separating outward minimising minimal sphere in $\Omega$.
\end{prop}
\begin{proof}
From the assumption and Theorem \ref{PI}, $V$ cannot be the boundary of a minimising hull, then we can conclude the proposition by Lemma \ref{MMSMH}. 
\end{proof}

Recall the definition of the Shi-Tam mass $m_{ST}(\Omega)$ whose properties are detailed in Section 2 Theorem 2.5 in \cite{ST3}.  

\begin{defn}
Let $\O_1\subset \subset \O_2\subset \O$ such that $\O_1$ and $\O_2$ have smooth boundaries. Let $d$ denote the distance between $\O_1$ and $\p \O_2$ and $\iota $ be the infimum of the injectivity radius of points in $\{ x\in \O_2 \, |\, d(x,\p \O_2)>\frac{d}{4} \}$. Let $\mathcal{F}_{\O_2}$ be the family of precompact connected minimising hulls with $C^2$ boundary in $\O_2$. Define 
\be
m(\O_1,\O_2)=\sup_{E\in \mathcal{F}_{\O_2},\, E\subset \O_1} m_H(\p E).
\ee
Then the Shi-Tam mass is defined by 
\be
m_{ST}(\O)=\sup \alpha_{\O_1,\O_2}m(\O_1,\O_2), 
\ee 
where the supremum is taken over all $\O_1\subset \subset \O_2\subset \O$ with smooth boundaries and 
\be
\alpha_{\O_1,\O_2}^2=\min\{ 1, \frac{2\pi \mathcal{K}^{-2}\int_0^r \tau^{-1}\sin(\mathcal{K} \tau)^2 d\tau }{|\p \O_1|} \}, 
\ee
where $r=\min\{ \frac{d}{2}, \iota \}$ and $\mathcal{K}>0$ is the upper bound of the sectional curvature of $\O_2$. 
\end{defn}

First, we have the following theorem. 
\begin{thm}\label{STW}
Let $(\O, g, k)$ be admissible. Suppose there is no outward minimising minimal sphere in $\O$, then $m_{ST}(\Omega)\leq \MW(\S)$. 
\end{thm}
\begin{proof}
Assume $\O$ has no outward minimising minimal sphere, then it is diffeomorphic to a ball in $\R^3$ by Theorem 1' in \cite{MSY} as $\S$ is mean convex. Let $\O_1 \subset \subset \O_2 \subset \O$ with smooth boundaries. Let $E$ be a connected precompact minimizing hull in $\O_2$ with $C^2$ boundary such that $E \subset  \O_1$. And we can assume $m_H(\p E)\geq 0$. By $\S$ being mean convex, the strictly minimising hull $E'$ of $E$ exists. And $E'$ is connected  with $C^{1,1}$ boundary since $E$ is connected with $C^2$ boundary. 
By Theorem \ref{PI}, we can replace the equation (3.12) in the proof of Theorem 3.2 in \cite{ST3} by 
$$\MW(\S)\geq m_H(\p E').$$
The result then follows from its subsequent arguments in \cite{ST3}.  
\end{proof}
We then have the following corollary. 
\begin{cor}
Let $(\Omega,g,k)$ be admissible. If either $m_{ST}(\O)>\MW(\S)$ or $m_{ST}(\O)\geq \frac{1}{2}diam(\S)$, then there is a separating
outward minimizing minimal sphere in $\O$, where $diam(\S)$ is the diameter of $\S$ with respect to the metric induced by $g$.
\end{cor}
\begin{proof}
The first part is just the contrapositive of Theorem \ref{STW}. For the second part, let $\rho$ be the radius of the smallest circumscribed ball of $\S$ in $\R^3$.  Then by the definition of $\MW(\S)$, the Minkowski integral formula (P.136 Lemma 6.29 of \cite{K}) and Gauss-Bonnet Theorem, we have the following, 
\be
\MW(\S)<\frac{1}{8\pi}\int_{\S}H_0 \leq \frac{\rho}{8\pi}\int_{\S} \frac{R_\S}{2} = \frac{\rho}{2} <\frac{diam(\S)}{2}\leq m_{ST}(\O).
\ee
By Theorem \ref{STW} again, the result follows.
\end{proof}

\subsection{Detection of apparent horizons}
As discussed in Section \ref{Connections to quasilocal mass}, if $H>|tr_{\S}k|$ on $\S$, then $\MW(\S)\geq m_{LY}(\S)$. Therefore the statements in \cite{ALY} (\cite{ST3}) can also be applied to $\MW(\S)$, i.e. $\MW(\Sigma)$ can be used to determine whether an apparent horizon (MOTS) with certain properties exists in $\Omega$ by the following localised spacetime Penrose inequality.    
\begin{cor}(cf.  \cite{ALY} Proposition 1.3)
Let $(\Omega,g,k)$, where $\p \Omega$ has positive Gauss curvature, be admissible in the sense of \cite{ALY} and $S$ be a closed surface in $\Omega$. Suppose that either
\begin{enumerate}
\item $S$ is a MOTS and its projection in blow up Jang graph of $\Omega$ is outward minimizing, or
\item $S$ has an outward minimizing minimal surface projection in a Jang graph of $\Omega$,
\end{enumerate}
then
$$\MW(\S)\geq\sqrt{\frac{|S|}{16\pi}}.$$
\end{cor}

\section{Non-existence of DEC fill-ins}\label{DECfillin}
First, in \cite{SWWZ},  the construction of a scalar flat and asymptotically flat metric and decreasing total mean curvature difference along the radial direction proved in \cite{ST} are the main inputs to show that under certain assumptions, if an NNSC fill-in, i.e. $(\Omega, g)$ with $R_g\geq0$, of Bartnik data $(\S, \gamma, H)$ exists, then there is a contradiction to the Riemannian positive mass theorem with corners (\cite{ST}, \cite{M1}).  Heuristically from the perspective of energy, by \eqref{Ham}, we can see that if the boundary energy is too large, then the gravitation contribution must be negative. Motivated by this,  with Theorem \ref{main} and \cite{Shibuya}, we can obtain similar results regarding the energy condition of the fill-in of a spacetime Bartnik data set.  In other words,  a partial answer to the spacetime version of Conjecture \ref{Gromov} can be obtained. 

\begin{defn}(cf. \cite{B2} Definition 2, \cite{SWWZ})
For $n\geq 3$, a tuple $(\Sigma^{n-1}, \gamma, \alpha, H, \beta)$ is called a spacetime Bartnik data set, where $(\Sigma, \gamma, \alpha)$ is an oriented closed null-cobordant initial data set with $\alpha\in C^{1,\alpha}$, while $H$ and $\beta$ are respectively a smooth function and a $C^{1,\alpha}$ 1-form on $\Sigma$. A compact initial data set $(\Omega^n,g,k)$ is called a fill-in of $D_{SB}$ if there is an isometry $\phi: (\Sigma^{n-1}, \gamma)\to (\p\Omega, g|_{\p\Omega}) $ such that 
\begin{enumerate}
\item $\phi^*H_g=H$,  where $H_g$ is the mean curvature of $\p\Omega$ to $g$ with respect to the outward unit normal $\nu$,  
\item $\phi^*tr_{\p \Omega}k=tr_{\S}\alpha$,  and
\item $\phi^*( k(\nu,\cdot))=\beta$.
\end{enumerate}
\end{defn}
We can see from the above definition that on $\S$,  $$\phi^{*}(|\omega|_g)=\phi^*(|\pi(\nu,\cdot)|_g) =\sqrt{(\tr_{\S}\alpha)^2+|\beta|_{\gamma}^2}.$$

\begin{defn}
A fill-in $(\Omega^n, g, k)$ is said to satisfy a topological assumption $(T)$ if either one of the following holds: 
\begin{enumerate}
\item for $n=3$, there exists $\MS$, a finite (possibly empty) disjoint union of connected weakly trapped surfaces, such that $H_2(\Omega_{ext}, \MS, \mathbb{Z})=0$, where $\Omega_{ext}$ denotes the portion of $\Omega$ outside $\MS$. 
\item for $n\geq3$,  $\Omega$ is spin.  
\end{enumerate}
\end{defn}

The condition $(T)$ allows applications of Theorem \ref{main} or \cite{Shibuya} Section VI. Following in this section, $n\geq3$. For notation simplicity, the isometric embedding $\phi$ is omitted after gluing and identification if without ambiguity. Moreover, $\gamma_{std}$ denotes the standard
metric on $\mathbb{S}^{n-1}$ induced from the Euclidean space. 

\begin{thm} (cf. \cite{SWWZ} Theorem 1.3)
Let $D_{SB}:=(\mathbb{S}^{n-1}, \gamma, \alpha, H, \beta )$ be a spacetime Bartnik data set. If $\gamma$ is isotopic to $\gamma_{std}$ in $\M^q_{psc}(\mathbb{S}^{n-1}):=\{\eta: C^q \text{\,\,metrics on\,\,} \mathbb{S}^{n-1} {\,\,with\,\,} R_{\eta}>0 \}$, where $q\geq 5$, then there exists a constant $h_0=h_0(n,\gamma)>0$ such that 
if 
$$H-f>0\,\,\, \text{and\,\,\,\,}\, \int_{\mathbb{S}^{n-1}} H-f \, d\mu_{\gamma}>h_0,$$
where $f:=\sqrt{(\tr_{\S}\alpha)^2+|\beta|_{\gamma}^2}$,
then $D_{SB}$ cannot admit a fill-in satisfying $(T)$ and the dominant energy condition. 
\end{thm}

\begin{proof}
Notice that, since $\gamma$ is isotopic to $\gamma_{std}$ in $\M^q_{psc}$ (as mentioned in \cite{SWWZ}, by \cite{CM} Proposition 2.1 and its proof, the path $\gamma_t$ can be assumed to be smooth),  \cite{SWWZ} Lemma 2.1 shows that we can construct an asymptotically flat exterior extension $(M_{+}=\mathbb{S}^{n-1}\times [1,\infty), \bar g =dr^2+\bar \gamma_r)$ of $(\S^{n}, \gamma=\bar \gamma_1)$, which is exactly Euclidean (i.e. $\gamma_r=r^{n-1}\gamma_{std}$), for $r\geq s_0=s_0(\gamma_t, \eps)$ for any $\eps>0$. The choice of $\eps$ is to be determined.   

\ 

In the proof of \cite{SWWZ} Lemma 2.1, we can see that $R_{\bar g}$ is bounded by some constants depending on $\eps>0$ to be determined and $\gamma_t$. For this $M_+$, since $R_{\bar g}$ is bounded, the argument for the solvability of the initial value problem (\cite{SWWZ} P.14 equation (7)) for $u$ such that $g_+=u^2(r)dr^2+\bar \gamma_r$ is scalar flat and asymptotically flat in \cite{ST} is also applicable . The initial value $u(1)>0$ is to be determined. Let $\bar H_r$ and $H_r^{+}$ denote the mean curvature for $\S_r:=\mathbb{S}^{n-1}\times \{r\}$ in $\bar g$ and $g_+$ respectively.  

\

By equation (1.6) in \cite{ST}, and further note that ($\Sigma\times[s_o,\infty), \bar g)$ is Euclidean, then Lemma 4.2 and Lemma 2.10 in \cite{ST} are applicable. Together with the fact that, $\Sigma_{s_0}$ is a standard sphere in $\bar g$,  we have 
\be
\begin{split}
c(n) E_{ADM}(g_+)\leq & \int_{\Sigma_{s_0}}\bar H_{s_0} -H^+_{s_0} d\mu_{\bar \gamma_{s_0}}\\
\leq & n (n-1)|\mathbb{S}^{n-1}| s_0^{n-2}- \int_{\Sigma_{s_0}}H^+_{s_0} d\mu_{\bar \gamma_{s_0}}
\end{split}
\ee

What remains is to link the quantity $\int_{\Sigma_{s_0}}H^+_{s_0} \, d\mu_{\bar \gamma_{s_0}}$ to $\int_{\mathbb{S}^{n-1}} H-\omega \, d\mu_{\gamma}$.  Following \cite{SWWZ} Section 3.1, we thus have
\be
\begin{split}
c(n) E_{ADM}(g_+)
\leq & n (n-1)|\mathbb{S}^{n-1}| s_0^{n-2}- s_0^{\frac{(n-2)(1-\eps)}{2}}\int_{\mathbb{S}^{n-1}}  H_1^+ d\mu_{\gamma}.
\end{split}
\ee

\

Up to this step, it is the same as in the proof of Theorem 1.3 in \cite{SWWZ} as the construction for the extension only concerns the metric $\gamma$. 

\

Again, by equation (1.6) in \cite{ST}, we have $H_1^+=\frac{\bar H_1}{u(1)}$. Note that $\bar H_1>0$ by Lemma 2.1 equation (2) of \cite{SWWZ} and by fixing a choice of $\eps<<1$. 
If we choose the initial value $u(1):=\frac{\bar H_1}{H-f}$, then we have, 
\be
\begin{split}
c(n) E_{ADM}(g_+)
\leq & n (n-1)|\mathbb{S}^{n-1}| s_0^{n-2}- s_0^{\frac{(n-2)(1-\eps)}{2}}\int_{\mathbb{S}^{n-1}} H-f\, d\mu_{\gamma}. 
\end{split}
\ee

Thus, if we choose $h_0=h_0(n,\gamma)= n (n-1)|\mathbb{S}^{n-1}| s_0^{n-2-\frac{(n-2)(1-\eps)}{2}}$, then $E_{ADM}(g_+)<0$. Let $(\Omega, g,k)$ be a fill-in which satisfies the assumptions of the proposition. Then, $(\Omega, g,k)$ with $(M_+,g_+,0)$, altogether is an initial data set with a corner $\mathbb{S}^{n-1}$ on which $$H_g-H_1^+-|\pi(\nu,\cdot)|_g=H-(H-f)-f=0,$$ and $M_+$ satisfies the dominant energy condition by construction. If $(\Omega, g, k)$ furthermore simultaneously satisfies $(T)$ and the dominant energy condition, by Theorem \ref{main} or \cite{Shibuya} Section VI., $E_{ADM}(g_+)\geq0$, contradiction arises. 
\end{proof}

\begin{thm} (cf. \cite{SWWZ} Theorem 1.4)
Let $D_{SB}:=(\mathbb{S}^{n-1}, \gamma, \alpha, H, \beta )$ be a spacetime Bartnik data set. If $\gamma\in\M^n_{c,d}:=\{\eta: C^{\infty} \text{\,\,metrics on\,\,} \mathbb{S}^{n-1} {\,\,with\,\,}\\ |Rm_{\eta}|\leq c, \, diam({\eta})\leq d, \, vol(\eta)=vol(\gamma_{std})  \}$, then there exists a constant $C_0(n,c,d)>0$ such that if 
$$H-f\geq C_0,$$
where $f:=\sqrt{(\tr_{\S}\alpha)^2+|\beta|_{\gamma}^2}$,
then $D_{SB}$ cannot admit a fill-in satisfying $(T)$ and the dominant energy condition.
\end{thm}

\begin{proof}
The construction of an asymptotically flat extension with a corner where mean curvatures match is done in \cite{SWWZ} Lemma 2.4 followed by Lemma 2.1. The solvability of the initial value problem for $u$ is again by \cite{ST}, where $u(1)=\frac{\bar H_1(n,c)}{C_0}$ while $C_0$ is to be determined. If $C_0>0$ is sufficiently big,  depending on the curvature of the extension constructed in \cite{SWWZ} Lemma 2.4 which depends on $n, c$ and $d$, then $0<u<1$ on $M_+=\mathbb{S}^{n-1}\times[1,\infty)$. Moreover, $E_{ADM}(g_+=u(r)^2+\bar \gamma_r)<0$. 

\

Assume on the contrary that there exists a fill-in of $D_{SB}$, $(\Omega,g,k)$ which satisfies the assumptions of the proposition, $(T)$ and the dominant energy condition. Gluing $\Omega$ and $M_+$, we have got an asymptotically flat initial data sets with 2 disjoint corners. For the corner in $M_+$, as mentioned, mean curvatures match by the construction in \cite{SWWZ} Lemma 2.4. For the corner $\p M_+=\Sigma_1$, we have $H-C_0-f\geq 0$. By Theorem \ref{main} or Section VI. in \cite{Shibuya}, $E_{ADM}(g_+)\geq 0$, contradiction arises. 
\end{proof}

While in \cite{SWW},  the parabolic method to extend metric in \cite{ST} is used to construct a PSC ($R_g>0$) collar, which combined with he Riemannian positive mass theorem with corners can show non-existence of NNSC fill-ins.  In the same spirit as above,  we can arrive at the following conclusion. 

\begin{thm} (cf. \cite{SWW} Theorem 1.2)
Let $D_{SB}:=(\S^{n-1}, \gamma, \alpha, H, \beta )$ be a spacetime Bartnik data set where $\S^{n-1}$ can be smoothly embedded into $\R^n$ and $\gamma$ is smooth. There exists a constant $C_0=C_0(\Sigma,\gamma)>0$ such that if $$H-f\geq C_0,$$ 
where $f:=\sqrt{(\tr_{\S}\alpha)^2+|\beta|_{\gamma}^2}$,
then $D_{SB}$ cannot admit a fill-in satisfying $(T)$ and the dominant energy condition.
\end{thm}

\begin{proof} Let $F:\Sigma^{n-1} \hookrightarrow \R^{n}$ be an embedding.  For $\lambda>0$, $\lambda F$ is also an embedding.  There exists a $\lambda_0>0$ such that $\gamma_1:=\lambda_0^2F^*(g_{Euc})>\gamma$, where $g_{Euc}$ is the Euclidean metric on $\R^n$.  Let $\tilde h$ denote the mean curvature of $\gamma_1$ with respect to the outward normal in $\R^n$.  Denote the unbounded region of $\R^n$ outside $\lambda_0 F (\Sigma)$ by $M_+$. 

\

By \cite{SWW} Lemma 2.1,  we know that there exists a cobordism $(\Sigma\times[0,1],\hat g)$ and $h_0, h_1\in C^{\infty}(\S)$ such that 
\begin{enumerate}
\item $\hat {g} |_{\Sigma \times\{0\}}=\gamma$ and $\hat{g}|_{\Sigma \times\{1\}}=\gamma_1$, 
\item With respect to $\hat{g}$ and the outward normal, the mean curvature of $\Sigma\times \{0\}$ and $\Sigma \times\{1\}$ are respectively $h_0$ and $h_1$, 
\item $h_1>\tilde h$ and
\item $R_{\hat g}>0$.  
\end{enumerate}

Pick $C_0=\max(-h_0)$. Let $(\Omega, g, k)$ be a fill-in satisfying the assumption of the proposition. Then glue $\p\Omega$ to $\Sigma\times [0,1]$ along $\Sigma\times \{0\}$ and further glue $\Sigma\times [0,1]$ along $\Sigma\times \{1\}$ to $M_+$.  Altogether, we have a manifold with a flat end and hence with $E_{ADM}=0$.  Across the corners $\Sigma\times \{0\}$ and $\Sigma\times \{1\}$, we respectively have $H-(-h_0)-f\geq 0$ and $h_1-\tilde h>0$. If $(\Omega,g,k)$ satisfies both $(T)$ and the dominant energy condition, by Theorem $\ref{main}$ or \cite{Shibuya} Section VI., $E_{ADM}>0$, contradiction arises. 
\end{proof}

\begin{rem}
For a charged initial data set $(M,g,\mE)$ with corners, we can consider charged harmonic functions (\cite{BHKKZ} Section 8) and the quantity associated with the divergence-free electric field $2\la \mE_{\pm},\nu \ra$ like $\pi_{\pm}(\nu, \cdot)$, in particular, $\mE$ are only required to be $L^{\infty}$ across the hypersurface where the corner of $g$ occurs.  The observations in this note are applicable and the corresponding results can be obtained. 
\end{rem}

\end{document}